\documentclass[
	english
]{scrartcl}

\usepackage[T1]{fontenc}
\usepackage[utf8]{inputenc}

\usepackage{amsmath,amsfonts,amsthm,amssymb}
\usepackage[colorlinks,citecolor=blue]{hyperref}
\usepackage{doi,natbib}
\usepackage{cancel}
\usepackage{relsize}

\usepackage[shortlabels]{enumitem}

\usepackage{changes}

\usepackage[figure]{hypcap}
\usepackage{graphicx}
\usepackage{subcaption} 

\usepackage{preamble}

\usepackage{marginnote}

\usepackage[capitalise,nameinlink]{cleveref}
\allowdisplaybreaks

\usepackage{tikz}
\usetikzlibrary{patterns}

\tikzset{
  schraffiert/.style={pattern=horizontal lines,pattern color=#1},
  schraffiert/.default=black
}

\newcommand\norm[2]{\left\Vert#1\right\Vert_{#2}}
\newcommand\abs[1]{\left\vert#1\right\vert}

\newcommand\N{\mathbb{N}}
\newcommand\R{\mathbb{R}}

\newcommand{\conv}{\operatorname{conv}}

\newcommand{\clconv}{\operatorname{\overline{\conv}}}

\renewcommand{\subseteq}{\subset}

\DeclareMathAlphabet{\mathpzc}{OT1}{pzc}{m}{it}

\newtheorem{theorem}{Theorem}[section]
\newtheorem{lemma}[theorem]{Lemma}
\newtheorem{proposition}[theorem]{Proposition}

\newtheorem{definition}[theorem]{Definition}
\newtheorem{example}[theorem]{Example}

\crefname{figure}{Figure}{Figures}

\begin{document}

\title{A comparison of first-order methods for the numerical solution of
	or-constrained optimization problems}
\author{%
	Patrick Mehlitz%
	\footnote{%
		Brandenburgische Technische Universität Cottbus--Senftenberg,
		Institute of Mathematics,
		03046 Cottbus,
		Germany,
		\email{mehlitz@b-tu.de},
		\url{https://www.b-tu.de/fg-optimale-steuerung/team/dr-patrick-mehlitz},
		ORCID: 0000-0002-9355-850X%
		}
	}

\publishers{}
\maketitle

\begin{abstract}
	 Mathematical programs with or-constraints form a new class of disjunctive optimization problems
	 with inherent practical relevance. In this paper, we provide a comparison of three different
	 first-order methods for the numerical treatment of this problem class which are inspired by
	 classical approaches from disjunctive programming.
	 First, we study the replacement of the or-constraints as nonlinear inequality constraints using
	 suitable NCP-functions. Second, we transfer the or-constrained program into a mathematical program
	 with switching or complementarity constraints which can be treated with the aid of well-known
	 relaxation methods. Third, a direct Scholtes-type relaxation of the or-constraints is investigated.
	 A numerical comparison of all these approaches which is based on three essentially different model
	 programs from or-constrained optimization closes the paper.
\end{abstract}


\begin{msc}	
	65K05, 90C30, 90C33
\end{msc}

\section{Introduction}\label{sec:introduction}

This paper is dedicated to the comparison of first-order methods for the numerical solution
of the or-constrained optimization problem
\begin{equation}\label{eq:MPOC}\tag{MPOC}
	\begin{aligned}
		f(x)&\,\rightarrow\,\min&&&\\
		g_i(x)&\,\leq\,0&\qquad&i\in\mathcal M&\\
		h_j(x)&\,=\,0&\qquad&j\in\mathcal P&\\
		G_l(x)\,\leq 0\,\lor\,H_l(x)&\,\leq\,0&\qquad&l\in\mathcal Q.&
	\end{aligned}
\end{equation}
Here, the functions 
$f,g_i,h_j,G_l,H_l\colon\R^n\to\R$  are assumed to be continuously differentiable for all 
$i\in\mathcal M:=\{1,\ldots,m\}$, 
$j\in\mathcal P:=\{1,\ldots,p\}$, and 
$l\in\mathcal Q:=\{1,\ldots,q\}$. 
For brevity, $g\colon\R^n\to\R^m$, $h\colon\R^n\to\R^p$, $G\colon\R^n\to\R^q$, and $H\colon\R^n\to\R^q$
are the mappings which possess the component functions 
$g_i$ ($i\in\mathcal M$), $h_j$ ($j\in\mathcal P$), $G_l$ ($l\in\mathcal Q$), 
and $H_l$ ($l\in\mathcal Q$), respectively.
Forthwith, the feasible set of \eqref{eq:MPOC}
will be denoted by $X\subset\R^n$.
Emphasizing that $\lor$ denotes the logical 'or',
the last $q$ constraints in \eqref{eq:MPOC} force $G_l(x)$ \emph{or} $H_l(x)$ to be less or equal to zero 
for all $l\in\mathcal Q$ whenever $x\in\R^n$ is feasible to \eqref{eq:MPOC}. Thus, we will refer to \eqref{eq:MPOC} as a 
\emph{mathematical program with or-constraints}. 

Clearly, \eqref{eq:MPOC} is an instance of logical mathematical programming, see e.g.\ \cite{Hooker2002} for an overview,
and covers several interesting applications e.g.\ from process engineering and scheduling, see 
\cite{Grossmann2002} and references therein. Particularly, or-constraints can be used to avoid
the formulation of so-called Big-$M$-constraints of type
\[
\begin{aligned}
	G_l(x)&\,\leq\, My_l&\quad&l\in\mathcal Q&\\
	H_l(x)&\,\leq\,M(1-y_l)&&l\in\mathcal Q&\\
	y_l&\,\in\,\{0,1\}&&l\in\mathcal Q&
\end{aligned}
\]
which would induce a mixed-integer-regime and the need for an a priori calculation of the constant $M>0$, see \cite{Mehlitz2019}.
Let us note that or-constrained programming with affine data functions is closely related to 
disjunctive programming in the sense of Balas, see \cite{Balas2018}, which means that a linear
function is minimized over the union of convex polyhedral sets. 
The situation where an arbitrary convex function
is minimized over the union of sets which are characterized via convex inequality constraints,
respectively, is discussed in \cite{GrossmannLee2003}. 
Using 
\begin{equation}\label{eq:def_O}
	O:=\{(a,b)\in\mathbb R^2\,|\,a\leq 0\,\lor\,b\leq 0\},
\end{equation} 
which can be written as a union of two 
convex polyhedral sets, \eqref{eq:MPOC} can be represented equivalently by
\[
	\begin{aligned}
		f(x)&\,\rightarrow\,\min&&&\\
		g_i(x)&\,\in\,\R^0_-&\qquad&i\in\mathcal M&\\
		h_j(x)&\,\in\,\{0\}&\qquad&j\in\mathcal P&\\
		(G_l(x),H_l(x))&\,\in\,O&\qquad&l\in\mathcal Q.&
	\end{aligned}
\]
which is a disjunctive program in the sense of \cite{BenkoGfrerer2018,FlegelKanzowOutrata2007}
where $\R^0_-$ denotes the set of all nonpositive real numbers.
As we will see later in \cref{sec:disjunctive_modification}, the model \eqref{eq:MPOC} is closely related to
so-called \emph{mathematical programs with switching constraints} (MPSCs), see \cite{KanzowMehlitzSteck2019,Mehlitz2019},
and \emph{mathematical programs with complementarity constraints} (MPCCs), see e.g.\
\cite{HoheiselKanzowSchwartz2013,LuoPangRalph1996,OutrataKocvaraZowe1998,Ye2005}.
First theoretical investigations which address the model \eqref{eq:MPOC} from the viewpoint of stationarity conditions
and constraint qualifications can be found in \cite[Section~7]{Mehlitz2019}.

This paper is devoted to the numerical treatment of \eqref{eq:MPOC}.
Here, we want to exploit three different ideas from disjunctive programming in order to develop numerical strategies 
for the computational solution of or-constrained programs.
We will focus our attention on the subsequently stated approaches:
\begin{itemize}
	\item reformulation of the or-constraints as nonlinear standard inequality constraints
		using so-called NCP-functions,
	\item reformulation of \eqref{eq:MPOC} as an MPSC or MPCC which can be tackled with the aid of
		relaxation methods from the literature, see e.g.\ \cite{KanzowMehlitzSteck2019, HoheiselKanzowSchwartz2013}, and
	\item direct relaxation of the or-constraints using a Scholtes-type method.
\end{itemize}
Here, we first study the individual qualitative properties of these methods before we provide a quantitative numerical
comparison based on different model problems from or-constrained optimization.

Originally, NCP-functions, where NCP abbreviates \emph{nonlinear complementarity program}, were introduced
in order to replace systems of complementarity constraints by nonlinear and possibly nonsmooth systems of equalities 
which can be solved e.g.\ by suitable Newton-type methods, see e.g.\ \cite{Fischer1992,Leyffer2006}. 
A satisfying overview of NCP-functions and their properties can be found in
\cite{Galantai2012,KanzowYamashitaFukushima1997,SunQi1999}. Here, we exploit the fact that some NCP-functions can be used to
replace or-constraints by nonlinear and possibly nonsmooth systems of inequalities.

As it turns out, one can transfer \eqref{eq:MPOC} into an MPSC or MPCC for the respective price of $2q$ slack variables.
It has been reported in \cite{Mehlitz2019} that this transformation generally comes along with additional local minimizers
of the surrogate MPSC and a similar behavior is at hand when the transformation into an MPCC is considered. 
Here, we additionally study the relationship of the stationary points associated with \eqref{eq:MPOC} and the stationary
points of the surrogate problem. As we will show, both transformations may induce additional stationary points which has
to be taken into account when this solution approach is exploited since first-order methods for the numerical solution
of MPSCs or MPCCs generally compute points satisfying certain problem-tailored stationarity conditions.

Noting that the variational geometry of $X$ is directly related to the variational geometry of the set $O$ from
\eqref{eq:def_O}, the irregularity of the kink in $O$ causes some irregularities in $X$ which may cause essential
trouble when \eqref{eq:MPOC} is solved numerically (e.g.\ via a direct treatment using suitable NCP-functions suggested above).
In order to overcome this potential issue, one could try to regularize this kink using a suitable relaxation approach. 
In the past, Scholtes' relaxation method has turned out to be a robust approach for the numerical solution of
MPCCs, MPSCs, and other models from disjunctive programming, see e.g.\ \cite{HoheiselKanzowSchwartz2013,KanzowMehlitzSteck2019,Scholtes2001}.
That is why we want to adapt it to the setting at hand. For this purpose, we suggest two different approaches based
on the smoothing of the popular Fischer--Burmeister function, see \cite{Fischer1992,Kanzow1996}, 
and a shifted version of the Kanzow--Schwartz function, see \cite{KanzowSchwartz2013}.

The remaining parts of the paper are organized as follows: In \cref{sec:notation}, we comment on the notation used
in the manuscript. Furthermore, we recall some basics from nonlinear programming as well as essential stationarity notions
and constraint qualifications for or-, switching-, and complementarity-constrained programming.
We study the reformulation of or-constraints with the aid of NCP-functions in \cref{sec:direct_treatment}. Additionally,
we investigate how the stationary points of \eqref{eq:MPOC} and its reformulation are related. As we will see, this
heavily depends on the choice of the underlying NCP-function. 
In \cref{sec:disjunctive_modification}, we present two reasonable reformulations of \eqref{eq:MPOC} as an MPSC as well as an
MPCC and study the relationship of the original problem and its surrogate w.r.t.\ minimizers and stationary points, respectively.
Direct Scholtes-type relaxation techniques associated with \eqref{eq:MPOC} which are based on the smoothed Fischer--Burmeister function
as well as the Kanzow--Schwartz function are the topic of \cref{sec:relaxation}. For both approaches, we study the
underlying convergence properties as well as the regularity of the appearing subproblems. 
In \cref{sec:numerics}, we present a quantitative comparison of all these methods based on three different models from
or-constrained programming, namely a nonlinear disjunctive program in the sense of Balas, see \cite{Balas2018}, an optimization
problem whose variables possess so-called gap domains, and an optimal control problem whose controls have to satisfy a 
pointwise or-constraint. Some concluding remarks close the paper in \cref{sec:concluding_remarks}.

\section{Notation and preliminaries}\label{sec:notation}

\subsection{Basic notation}

The subsequently stated tools from variational analysis can be found in 
\cite{Clarke1983,Mordukhovich2006,RockafellarWets1998}.

For a vector $x\in\R^n$, $\norm{x}{2}:=\sqrt{x\cdot x}$ is used to denote its Euclidean norm
where $\cdot$ represents the Euclidean product.
Choosing $\bar x\in\R^n$ and $\varepsilon>0$ arbitrarily, $\mathbb B^\varepsilon(\bar x)$ 
and $\mathbb S^\varepsilon(\bar x)$ denote the closed $\varepsilon$-ball and the
$\varepsilon$-sphere around $\bar x$, respectively.
Let $\mathtt e^n\in\R^n$ be the all-ones-vector, while for each $i\in\{1,\ldots,n\}$, 
$\mathtt e^n_i\in\R^n$ represents the $i$-th unit vector in $\R^n$.
Whenever $A\subset\R^n$ is a nonempty set and $\bar x\in A$ is chosen arbitrarily, then
the closed cone
\[
		\mathcal T_A(\bar x):=\left\{
				d\in\R^n\,\middle|\,
					\begin{aligned}
						&\exists\{x_k\}_{k\in\N}\subset A\,\exists\{\tau_k\}_{k\in\N}\subset\R_+\colon\\
						&\qquad x_k\to \bar x,\,\tau_k\downarrow 0,\,(x_k-\bar x)/\tau_k\to d
					\end{aligned}
			\right\}
\]
is called tangent or Bouligand cone to $A$ at $\bar x$ where $\R_+$ denotes the set of all positive real numbers.
Furthermore, the nonempty, closed, convex cone
\[
	A^\circ:=\{y\in\R^n\,|\,\forall x\in A\colon\, x\cdot y\leq 0\}
\]
is referred to as the polar cone of $A$. It is well known that for any two sets $B_1,B_2\subset\R^n$,
the polarization rule $(B_1\cup B_2)^\circ=B_1^\circ\cap B_2^\circ$ is valid.

Let $\{v^i\}_{i=1}^r,\{w^j\}_{j=1}^s\subset\R^n$ be two families of vectors. 
Then, $\{v^i\}_{i=1}^r\cup\{w^j\}_{j=1}^s$ is said to be positive-linearly dependent if there exist
scalars $\alpha_i\geq 0$ ($i=1,\ldots,r$) and $\beta_j$ ($j=1,\ldots,s$) which satisfy
\[
	0=\mathsmaller\sum\nolimits_{i=1}^r\alpha_iv^i+\mathsmaller\sum\nolimits_{j=1}^s\beta_jw^j
\]
and do not vanish at the same time. 
Supposing that such scalars do not exist, we call $\{v^i\}_{i=1}^r\cup\{w^j\}_{j=1}^s$
positive-linearly independent. Note that positive-linear independence is stable under small
perturbations, see \cite[Lemma~2.2]{KanzowMehlitzSteck2019}.

For a locally Lipschitz continuous functional $\psi\colon\R^n\to\R$ and some point $\bar x\in\R^n$, the (possibly empty)
set 
\[
	\partial^\textup{F}\psi(\bar x)
	:=
	\left\{
		y\in\R^n\,\middle|\,
			\liminf\limits_{x\to\bar x}
				\frac{\psi(x)-\psi(\bar x)-y\cdot(x-\bar x)}{\norm{x-\bar x}{2}}\geq 0
	\right\}
\]
is called Fr\'{e}chet (or regular) subdifferential of $\psi$ at $\bar x$. Based on that, one can define
the Mordukhovich (or basic) subdifferential of $\psi$ at $\bar x$ by
\[
	\partial^\textup{M}\psi(\bar x)
	:=
	\left\{
		y\in\R^n\,\middle|\,
			\begin{aligned}
			&\exists\{x_k\}_{k\in\N}\subset\R^n\,\exists\{y_k\}_{k\in\N}\subset\R^n\colon\\
			&\qquad	x_k\to\bar x,\,y_k\to y,\,y_k\in\partial^\textup{F}\psi(x_k)\,\forall k\in\N
			\end{aligned}
	\right\}.
\]
Finally, the Clarke (or convexified) subdifferential of $\psi$ at $\bar x$ is given by
\[
	\partial^\textup{C}\psi(\bar x):=\clconv\partial^\textup{M}\psi(\bar x)
\]
where $\clconv A$ denotes the closed, convex hull of $A\subset\R^n$.
By construction, we have $\partial^\textup F\psi(\bar x)\subset\partial^\textup{M}\psi(\bar x)\subset\partial^\textup{C}\psi(\bar x)$
and all these sets coincide with the singleton comprising only the gradient of $\psi$ at $\bar x$ whenever
$\psi$ is continuously differentiable at $\bar x$.

\subsection{Preliminaries from nonlinear programming}

Here, we briefly recall basic constraint qualifications from nonlinear programming which can be found in \cite{BazaraaSheraliShetty1993}.
Therefore, let us consider the nonlinear program
\begin{equation}\label{eq:NLP}\tag{NLP}
	\begin{aligned}
		f(x)&\,\to\,\min&&&\\
		g_i(x)&\,\leq\,0&\qquad&i\in\mathcal M&\\
		h_j(x)&\,=\,0&&j\in\mathcal P,&
	\end{aligned}
\end{equation}
i.e.\ we leave the or-constraints in \eqref{eq:MPOC} out of our consideration for a moment.
Let $\tilde X\subset\R^n$ be the feasible set of \eqref{eq:NLP} and fix some point $\bar x\in \tilde X$.
Frequently, we will use the index set of active inequality constraints given by
\[
	I^g(\bar x):=\{i\in\mathcal M\,|\,g_i(\bar x)=0\}.
\]
The linearization cone to $\tilde X$ at $\bar x$ is given by
\[
	\mathcal L_{\tilde X}(\bar x):=
		\left\{
			d\in\R^n\,\middle|\,
				\begin{aligned}
					\nabla g_i(\bar x)\cdot d&\,\leq\,0&&i\in I^g(\bar x)\\
					\nabla h_j(\bar x)\cdot d&\,=\,0&&j\in\mathcal P
				\end{aligned}
		\right\}
\]
and its is well known that $\mathcal T_{\tilde X}(\bar x)\subset\mathcal L_{\tilde X}(\bar x)$ is valid
while the converse inclusion only holds true under validity of a constraint qualification in general.
Let us note that the polar cone of $\mathcal L_{\tilde X}(\bar x)$ is given by
\[
	\mathcal L_{\tilde X}(\bar x)^\circ
	=
	\left\{
		\mathsmaller\sum\nolimits_{i\in I^g(\bar x)}\lambda_i\nabla g_i(\bar x)
		+\mathsmaller\sum\nolimits_{j\in\mathcal P}\rho_j\nabla h_j(\bar x)\,\middle|\,
			\forall i\in I^g(\bar x)\colon\,\lambda_i\geq 0
	\right\}.
\]
Now, recall that LICQ (MFCQ), the \emph{linear independence constraint qualification} 
(the \emph{Mangasarian--Fromovitz constraint qualification}),
holds true for \eqref{eq:NLP} at $\bar x$ whenever the vectors from
\[
	\{\nabla g_i(\bar x)\,|\,i\in I^g(\bar x)\}\cup\{\nabla h_j(\bar x)\,|\,j\in\mathcal P\}
\]
are linearly independent (positive-linearly independent). Furthermore, GCQ, the \emph{Guignard constraint qualification},
is valid at $\bar x$ whenever $\mathcal T_{\tilde X}(\bar x)^\circ=\mathcal L_{\tilde X}(\bar x)^\circ$ holds true.
Clearly, we have
\[
	\textup{LICQ}\quad\Longrightarrow\quad\text{MFCQ}\quad\Longrightarrow\quad\text{GCQ}
\]
and the validity of any of these constraint qualifications at a local minimizer $\bar x$ of \eqref{eq:NLP} implies
that the latter is a Karush--Kuhn--Tucker (KKT) point of \eqref{eq:NLP}, i.e.\ there are multipliers
$\lambda_i\geq 0$ ($i\in I^g(\bar x)$) and $\rho_j$ ($j\in\mathcal P$) which satisfy
\[
	0=\nabla f(\bar x)+\sum\limits_{i\in I^g(\bar x)}\lambda_i\nabla g_i(\bar x)
		+\sum\limits_{j\in\mathcal P}\rho_j\nabla h_j(\bar x).
\]
Let us briefly mention that nonsmooth multiplier rules of KKT-type for the problem \eqref{eq:NLP}
with locally Lipschitz continuous but not necessarily differentiable data functions
which are stated in terms of Mordukhovich's or Clarke's subdifferential can be found in 
\cite{Mordukhovich2006} and \cite{Vinter2000}, respectively.

\subsection{Preliminaries from disjunctive programming}

In this section, we briefly recall some stationarity conditions and constraint qualifications for
three classes of disjunctive programs namely MPOCs, MPSCs, and MPCCs.
\subsubsection{Mathematical programs with or-constraints}

Problems of type \eqref{eq:MPOC} were considered from the viewpoint of disjunctive programming
in \cite[Section~7]{Mehlitz2019} first. In the latter paper, the author introduced reasonable
stationarity notions and constraint qualifications for \eqref{eq:MPOC}.

Let us fix a feasible point $\bar x\in X$ of \eqref{eq:MPOC}. 
Frequently, we will make use of the index sets defined below:
\begin{align*}
	I^{-0}(\bar x)&:=\{l\in\mathcal Q\,|\,G_l(\bar x)<0\land H_l(\bar x)=0\}&
	I^{0-}(\bar x)&:=\{l\in\mathcal Q\,|\,G_l(\bar x)=0 \land H_l(\bar x)<0\}&\\
	I^{-+}(\bar x)&:=\{l\in\mathcal Q\,|\,G_l(\bar x)<0 \land H_l(\bar x)>0\}&
	I^{+-}(\bar x)&:=\{l\in\mathcal Q\,|\,G_l(\bar x)>0 \land H_l(\bar x)<0\}&\\
	I^{0+}(\bar x)&:=\{l\in\mathcal Q\,|\,G_l(\bar x)=0 \land H_l(\bar x)>0\}&
	I^{+0}(\bar x)&:=\{l\in\mathcal Q\,|\,G_l(\bar x)>0 \land H_l(\bar x)=0\}&\\
	I^{--}(\bar x)&:=\{l\in\mathcal Q\,|\,G_l(\bar x)<0 \land H_l(\bar x)<0\}&
	I^{00}(\bar x)&:=\{l\in\mathcal Q\,|\,G_l(\bar x)=0 \land H_l(\bar x)=0\}.&
\end{align*}
Clearly, these index sets provide a disjoint partition of $\mathcal Q$.
Furthermore, we set
\begin{equation}\label{eq:def_I}
	\mathcal I(\bar x):=I^{0+}(\bar x)\cup I^{+0}(\bar x)\cup I^{00}(\bar x)
\end{equation}
for brevity. Now, we are in position to state definitions of MPOC-tailored stationarity
concepts.
\begin{definition}\label{def:stationarities:MPOC}
	Let $\bar x\in X$ be a feasible point of \eqref{eq:MPOC}.
	Then, $\bar x$ is said to be
	\begin{enumerate}
		\item \emph{weakly stationary} (W-stationary) for \eqref{eq:MPOC} whenever there exist multipliers
		$\lambda_i\geq 0$ ($i\in I^g(\bar x)$), $\rho_j$ ($j\in\mathcal P$),
		$\mu_l\geq 0$ ($l\in I^{0+}(\bar x)\cup I^{00}(\bar x)$), and 
		$\nu_l\geq 0$ ($l\in I^{+0}(\bar x)\cup I^{00}(\bar x)$) which satisfy
		\begin{equation}\label{eq:WSt}
			\begin{split}
				&0=\nabla f(\bar x)
					+\sum\limits_{i\in I^g(\bar x)}\lambda_i\nabla g_i(\bar x)
					+\sum\limits_{j\in\mathcal P}\rho_j\nabla h_j(\bar x)\\
				&\qquad 
					+\sum\limits_{l\in I^{0+}(\bar x)\cup I^{00}(\bar x)}\mu_l\nabla G_l(\bar x)
					+\sum\limits_{l\in I^{+0}(\bar x)\cup I^{00}(\bar x)}\nu_l\nabla H_l(\bar x),	
			\end{split}
		\end{equation}
		\item \emph{Mordukhovich-stationary} (M-stationary) for \eqref{eq:MPOC} whenever it is W-stationary
		while the associated multipliers additionally satisfy
		\[
			\forall l\in I^{00}(\bar x)\colon\quad \mu_l\nu_l=0,
		\]
		\item \emph{strongly stationary} (S-stationary) for \eqref{eq:MPOC} whenever it is W-stationary 
		while the associated multipliers additionally satisfy
		\[
			\forall l\in I^{00}(\bar x)\colon\quad \mu_l=0\,\land\,\nu_l=0.
		\]
	\end{enumerate}
\end{definition}

By definition, we obtain the relations
\[
	\text{S-stationarity}\quad\Longrightarrow\quad\text{M-stationarity}\quad
	\Longrightarrow\quad\text{W-stationarity}
\]
between these stationarity notions which are visualized in \cref{fig:St_MPOC}.
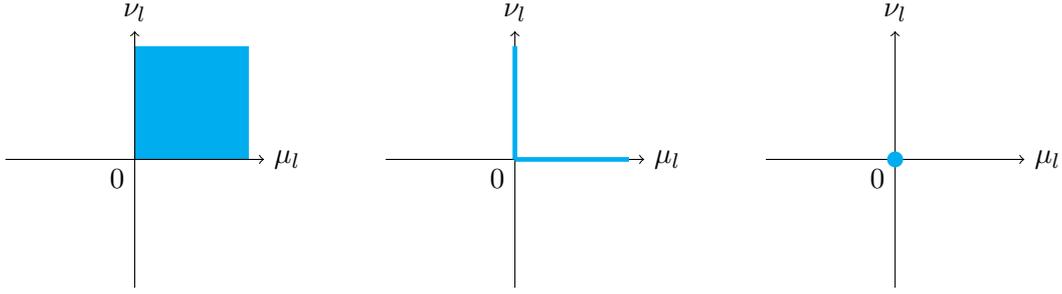
\begin{figure}[h]\centering
  \begin{tikzpicture}
  \fill[color=cyan] (0,1.5)--(1.5,1.5)--(1.5,0)--(0,0);  
  \draw[->] (-1.7,0) -- (1.7,0) node[right] {$\mu_l$};
  \draw[->] (0,-1.7) -- (0,1.7) node[above] {$\nu_l$};
  \node[below left] at (0,0){$0$};
  \draw[->] (3.3,0) -- (6.7,0) node[right] {$\mu_l$};
  \draw[->] (5,-1.7) -- (5,1.7) node[above] {$\nu_l$};
  \draw[very thick, color = cyan] (5,0.01) -- (6.5,0.01);
  \draw[very thick, color = cyan] (5,-0.01) -- (6.5,-0.01);
  \draw[very thick, color = cyan] (5.01,0) -- (5.01,1.5);
  \draw[very thick, color = cyan] (4.99,0) -- (4.99,1.5);
  \node[below left] at (5,0){$0$}; 
  \draw[->] (8.3,0) -- (11.7,0) node[right] {$\mu_l$};
  \draw[->] (10.0,-1.7) -- (10.0,1.7) node[above] {$\nu_l$};  
  \fill[color=cyan] (10,0) circle [radius=3pt];
  \node[below left] at (10,0){$0$};
\end{tikzpicture}
\caption{Geometric visualizations of W-, M-, and S-stationarity for 
	the program \eqref{eq:MPOC} w.r.t.\ an index $l\in I^{00}(\bar x)$.}
\label{fig:St_MPOC}
\end{figure}
We say that MPOC-LICQ (MPOC-MFCQ) is valid at a feasible point $\bar x\in X$ of
\eqref{eq:MPOC} whenever the gradients from
\[
	\begin{split}
		\Bigl[\{\nabla g_i(\bar x)\,|\,i\in I^g(\bar x)\}
			&\cup\{\nabla G_l(\bar x)\,|\,l\in I^{0+}(\bar x)\cup I^{00}(\bar x)\}\\
			&\cup\{\nabla H_l(\bar x)\,|\,l\in I^{+0}(\bar x)\cup I^{00}(\bar x)\}
		\Bigr]\cup\{\nabla h_j(\bar x)\,|\,j\in\mathcal P\}
	\end{split}
\]
are linearly independent (positive-linearly independent).
It has been shown in \cite[Theorem~7.1]{Mehlitz2019} that a local minimizer
of \eqref{eq:MPOC} where MPOC-LICQ holds is an S-stationary point.
By means of standard arguments, it is easy to confirm that the validity of
MPOC-MFCQ only yields M-stationarity of local minimizers of \eqref{eq:MPOC}
in general.

\subsubsection{Mathematical programs with switching constraints}

Let us consider the mathematical program
\begin{equation}\label{eq:MPSC}\tag{MPSC}
	\begin{aligned}
		f(x)&\,\to\,\min&&&\\
		g_i(x)&\,\leq\,0&\qquad&i\in\mathcal M&\\
		h_j(x)&\,=\,0&&j\in\mathcal P&\\
		\tilde G_l(x)\tilde H_l(x)&\,=\,0&&l\in\mathcal Q&
	\end{aligned}
\end{equation}
where $\tilde G_l,\tilde H_l\colon\R^n\to\R$ ($l\in\mathcal Q$)
are continuously differentiable functions.
For later use, let $\tilde G,\tilde H\colon\R^n\to\R^q$ be the maps
possessing the component functions $\tilde G_l$ ($l\in\mathcal Q$) and
$\tilde H_l$ ($l\in\mathcal Q$), respectively.
Note that \eqref{eq:MPSC} results from \eqref{eq:MPOC} by replacing
the $q$ or-constraints by $q$ so-called switching constraints.
That is why we refer to \eqref{eq:MPSC} as a \emph{mathematical
program with switching constraints}. Theoretical and numerical investigations
which address this problem class as well as an overview of underlying 
applications can be found in the recent papers
\cite{KanzowMehlitzSteck2019,Mehlitz2019}.

Let $X_\textup{SC}\subset\R^n$ be the feasible set of \eqref{eq:MPSC}
and fix some point $\bar x\in X_\textup{SC}$. We define 
\begin{align*}
	I^{\tilde G}(\bar x)&:=\{l\in\mathcal Q\,|\,\tilde G_l(\bar x)=0\,\land\,\tilde H_l(\bar x)\neq 0\},\\
	I^{\tilde H}(\bar x)&:=\{l\in\mathcal Q\,|\,\tilde G_l(\bar x)\neq 0\,\land\,\tilde H_l(\bar x)= 0\},\\
	I^{\tilde G\tilde H}(\bar x)&:=\{l\in\mathcal Q\,|\,\tilde G_l(\bar x)=0\,\land\,\tilde H_l(\bar x)= 0\}.
\end{align*}
Clearly, these sets provide a disjoint partition of $\mathcal Q$ and allow
us to state suitable problem-tailored stationarity notions for \eqref{eq:MPSC}.
\begin{definition}\label{def:St_MPSC}
	Let $\bar x\in X_\textup{SC}$ be a feasible point of \eqref{eq:MPSC}.
	Then, $\bar x$ is said to be
	\begin{enumerate}
	\item \emph{weakly stationary} (W$_\textup{SC}$-stationary) for 
	\eqref{eq:MPSC} whenever there exist multipliers $\lambda_i\geq 0$
	($i\in I^g(\bar x)$), $\rho_j$ ($j\in\mathcal P$),
	$\tilde \mu_l$ ($l\in I^{\tilde G}(\bar x)\cup I^{\tilde G\tilde H}(\bar x)$), and
	$\tilde \nu_l$ ($l\in I^{\tilde H}(\bar x)\cup I^{\tilde G\tilde H}(\bar x)$) which satisfy
	\[
		\begin{split}
			&0=\nabla f(\bar x)
				+\sum\limits_{i\in I^g(\bar x)}\lambda_i\nabla g_i(\bar x)
				+\sum\limits_{j\in\mathcal P}\rho_j\nabla h_j(\bar x)\\
			&\qquad
				+\sum\limits_{l\in I^{\tilde G}(\bar x)\cup I^{\tilde G\tilde H}(\bar x)}
					\tilde \mu_l\nabla \tilde G_l(\bar x)
				+\sum\limits_{l\in I^{\tilde H}(\bar x)\cup I^{\tilde G\tilde H}(\bar x)}
					\tilde \nu_l\nabla \tilde H_l(\bar x),
		\end{split}
	\]
	\item \emph{Mordukhovich-stationary} (M$_\textup{SC}$-stationary) for \eqref{eq:MPSC} whenever
	it is W$_\textup{SC}$-stationary while the associated multipliers additionally satisfy
	\[
		\forall l\in I^{\tilde G\tilde H}(\bar x)\colon\quad \tilde \mu_l\tilde \nu_l=0,
	\]
	\item \emph{strongly stationary} (S$_\textup{SC}$-stationary) for \eqref{eq:MPSC} whenever
	it is W$_\textup{SC}$-stationary while the associated multipliers additionally satisfy
	\[
		\forall l\in I^{\tilde G\tilde H}(\bar x)\colon\quad \tilde \mu_l=0\,\land\,\tilde \nu_l=0.
	\]
	\end{enumerate}
\end{definition}

Again, we obtain the relations
\[
	\text{S$_\textup{SC}$-stationary}\quad\Longrightarrow\quad
	\text{M$_\textup{SC}$-stationary}\quad\Longrightarrow\quad
	\text{W$_\textup{SC}$-stationary}
\]
by definition of these stationarity concepts. 
In order to avoid confusion, we used the index SC to emphasize that the above stationarity notions
address \eqref{eq:MPSC} although this is also quite clear from the context.
It should be noted that S$_\textup{SC}$-stationarity is equivalent to the KKT conditions
of \eqref{eq:MPSC} where the switching constraints are interpreted as simple equality constraints.
A visualization of these stationarity concepts is provided in \cref{fig:St_MPSC}.
\begin{figure}[h]\centering
  \begin{tikzpicture}
  \fill[color=cyan] (-1.5,-1.5)--(-1.5,1.5)--(1.5,1.5)--(1.5,-1.5);  
  \draw[->] (-1.7,0) -- (1.7,0) node[right] {$\tilde \mu_l$};
  \draw[->] (0,-1.7) -- (0,1.7) node[above] {$\tilde \nu_l$};
  \node[below left] at (0,0){$0$};
  \draw[->] (3.3,0) -- (6.7,0) node[right] {$\tilde \mu_l$};
  \draw[->] (5,-1.7) -- (5,1.7) node[above] {$\tilde \nu_l$};
  \draw[very thick, color = cyan] (3.5,0.01) -- (6.5,0.01);
  \draw[very thick, color = cyan] (3.5,-0.01) -- (6.5,-0.01);
  \draw[very thick, color = cyan] (5.01,-1.5) -- (5.01,1.5);
  \draw[very thick, color = cyan] (4.99,-1.5) -- (4.99,1.5);
  \node[below left] at (5,0){$0$}; 
  \draw[->] (8.3,0) -- (11.7,0) node[right] {$\tilde \mu_l$};
  \draw[->] (10.0,-1.7) -- (10.0,1.7) node[above] {$\tilde \nu_l$};  
  \fill[color=cyan] (10,0) circle [radius=3pt];
  \node[below left] at (10,0){$0$};
\end{tikzpicture}
\caption{Geometric visualizations of W$_\textup{SC}$-, M$_\textup{SC}$-, and S$_\textup{SC}$-stationarity for 
	the program \eqref{eq:MPSC} w.r.t.\ an index $l\in I^{\tilde G\tilde H}(\bar x)$.}
\label{fig:St_MPSC}
\end{figure}
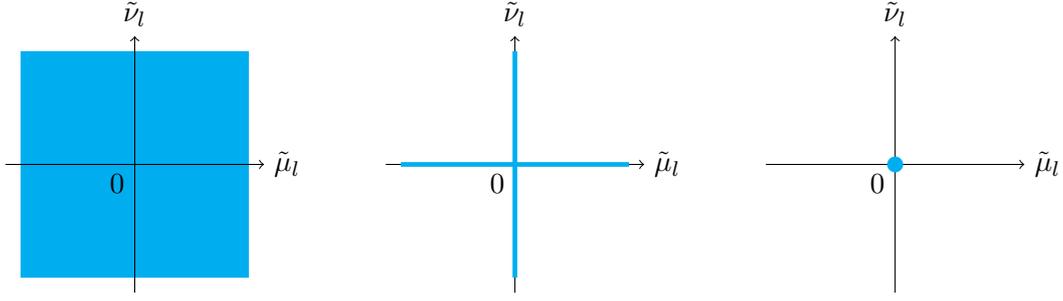

\subsubsection{Mathematical programs with complementarity constraints}

Finally, we would like to mention so-called \emph{mathematical programs with complementarity constraints}
which are optimization problems of the following type:
\begin{equation}\label{eq:MPCC}\tag{MPCC}
	\begin{aligned}
		f(x)&\,\to\,\min&&&\\
		g_i(x)&\,\leq\,0&\qquad&i\in\mathcal M&\\
		h_j(x)&\,=\,0&&j\in\mathcal P&\\
		0\,\leq\,\bar G_l(x)\,\perp\,\bar H_l(x)&\,\geq\,0&&l\in\mathcal Q.&
	\end{aligned}
\end{equation}
Therein, the last $q$ constraints, which are stated w.r.t.\ continuously differentiable functions
$\bar G,\bar H\colon\R^n\to\R^q$ whose components will be addressed by $\bar G_1,\ldots,\bar G_q\colon\R^n\to\R$ and
$\bar H_1,\ldots,\bar H_q\colon\R^n\to\R$, respectively, induce a complementarity regime
since they demand that for each $l\in\mathcal Q$, $G_l(x)$ and $H_l(x)$ are nonnegative while at least one of
those numbers needs to vanish for each feasible point $x\in\R^n$ of \eqref{eq:MPCC}.
During the last decades, complementarity-constrained optimization has been considered eagerly from a
theoretical and numerical point of view due to numerous underlying applications, see e.g.\ 
\cite{LuoPangRalph1996,OutrataKocvaraZowe1998}.

We exploit $X_\textup{CC}\subset\R^n$ in order to denote the feasible set of \eqref{eq:MPCC}.
Let us fix a feasible point $\bar x\in X_\textup{CC}$. Then, the index sets
\begin{align*}
	I^{0+}_\textup{CC}(\bar x)&:=
		\{l\in\mathcal Q\,|\,\bar G_l(\bar x)=0\,\land\,\bar H_l(\bar x)>0\},\\
	I^{+0}_\textup{CC}(\bar x)&:=
		\{l\in\mathcal Q\,|\,\bar G_l(\bar x)>0\,\land\,\bar H_l(\bar x)=0\},\\
	I^{00}_\textup{CC}(\bar x)&:=
		\{l\in\mathcal Q\,|\,\bar G_l(\bar x)=0\,\land\,\bar H_l(\bar x)=0\}
\end{align*}
provide a disjoint partition of $\mathcal Q$. Note that we used the index CC in order to
distinguish the above index sets from their respective counterparts which are related to
\eqref{eq:MPOC}. Next, we recall some stationarity notions from complementarity-constrained
programming, see e.g.\ \cite{Ye2005}. Again, we use the index CC in order to emphasize
that the stationarity notions of interest are related to \eqref{eq:MPCC}.
\begin{definition}\label{def:St_MPCC}
	Let $\bar x\in X_\textup{CC}$ be a feasible point of \eqref{eq:MPCC}.
	Then, $\bar x$ is said to be
	\begin{enumerate}
	\item \emph{weakly stationary} (W$_\textup{CC}$-stationary) for 
	\eqref{eq:MPCC} whenever there exist multipliers $\lambda_i\geq 0$
	($i\in I^g(\bar x)$), $\rho_j$ ($j\in\mathcal P$),
	$\bar \mu_l$ ($l\in I^{0+}_\textup{CC}(\bar x)\cup I^{00}_\textup{CC}(\bar x)$), and
	$\bar \nu_l$ ($l\in I^{+0}_\textup{CC}(\bar x)\cup I^{00}_\textup{CC}(\bar x)$) which satisfy
	\[
		\begin{split}
			&0=\nabla f(\bar x)
				+\sum\limits_{i\in I^g(\bar x)}\lambda_i\nabla g_i(\bar x)
				+\sum\limits_{j\in\mathcal P}\rho_j\nabla h_j(\bar x)\\
			&\qquad
				-\sum\limits_{l\in I^{0+}_\textup{CC}(\bar x)\cup I^{00}_\textup{CC}(\bar x)}
					\bar \mu_l\nabla \bar G_l(\bar x)
				-\sum\limits_{l\in I^{+0}_\textup{CC}(\bar x)\cup I^{00}_\textup{CC}(\bar x)}
					\bar \nu_l\nabla \bar H_l(\bar x),
		\end{split}
	\]
	\item \emph{Clarke-stationary} (C$_\textup{CC}$-stationary) for \eqref{eq:MPCC} whenever
	it is W$_\textup{CC}$-stationary while the associated multipliers additionally satisfy
	\[
		\forall l\in I^{00}_\textup{CC}(\bar x)\colon\quad \bar\mu_l\bar\nu_l\geq 0,
	\]
	\item \emph{Mordukhovich-stationary} (M$_\textup{CC}$-stationary) for \eqref{eq:MPCC} whenever
	it is W$_\textup{CC}$-stationary while the associated multipliers additionally satisfy
	\[
		\forall l\in I^{00}_\textup{CC}(\bar x)\colon\quad \bar \mu_l\bar \nu_l=0\,\lor\,(\bar\mu_l>0\,\land\,\bar\nu_l>0),
	\]
	\item \emph{strongly stationary} (S$_\textup{CC}$-stationary) for \eqref{eq:MPCC} whenever
	it is W$_\textup{CC}$-stationary while the associated multipliers additionally satisfy
	\[
		\forall l\in I^{00}_\textup{CC}(\bar x)\colon\quad \bar \mu_l\geq 0\,\land\,\bar \nu_l\geq 0.
	\]
	\end{enumerate}
\end{definition}

By definition, we have
\[	\begin{split}
	\text{S$_\textup{CC}$-stationary}\quad&\Longrightarrow\quad
	\text{M$_\textup{CC}$-stationary}\\ &\Longrightarrow\quad
	\text{C$_\textup{CC}$-stationary}\quad\Longrightarrow\quad
	\text{W$_\textup{CC}$-stationary}.
	\end{split}
\]
Furthermore, one can check that the S$_\textup{CC}$-stationarity conditions of \eqref{eq:MPCC}
are equivalent to the KKT conditions of the equivalent NLP-model associated with \eqref{eq:MPCC}
where the complementarity constraints are restated as
\[
	\begin{aligned}
	\bar G_l(x)&\,\geq\, 0&\qquad&l\in\mathcal Q&\\
	\bar H_l(x)&\,\geq\, 0&&l\in\mathcal Q&\\
	\bar G(x)\cdot \bar H(x)&\,=\,0.&&&
	\end{aligned}
\]
All introduced stationarity notions are visualized in \cref{fig:St_MPCC}.
\begin{figure}[h]\centering
\scalebox{.75}{  \begin{tikzpicture}
  \fill[color=cyan] (-6.5,-1.5)--(-6.5,1.5)--(-3.5,1.5)--(-3.5,-1.5); 
  \draw[->] (-6.7,0) -- (-3.3,0) node[right] {$\bar \mu_l$};
  \draw[->] (-5,-1.7) -- (-5,1.7) node[above] {$\bar \nu_l$};
  \node[below left] at (-5,0){$0$};
  \fill[color=cyan] (0,1.5)--(1.5,1.5)--(1.5,0)--(0,0); 
  \fill[color=cyan] (0,0)--(-1.5,0)--(-1.5,-1.5)--(0,-1.5); 
  \draw[->] (-1.7,0) -- (1.7,0) node[right] {$\bar \mu_l$};
  \draw[->] (0,-1.7) -- (0,1.7) node[above] {$\bar \nu_l$};
  \node[below left] at (0,0){$0$};
  \fill[color=cyan] (5,0)--(6.5,0)--(6.5,1.5)--(5,1.5);
  \draw[->] (3.3,0) -- (6.7,0) node[right] {$\bar \mu_l$};
  \draw[->] (5,-1.7) -- (5,1.7) node[above] {$\bar \nu_l$};
  \draw[very thick, color = cyan] (3.5,0.01) -- (5,0.01);
  \draw[very thick, color = cyan] (3.5,-0.01) -- (5,-0.01);
  \draw[very thick, color = cyan] (5.01,-1.5) -- (5.01,0);
  \draw[very thick, color = cyan] (4.99,-1.5) -- (4.99,0);
  \node[below left] at (5,0){$0$};
  \fill[color=cyan] (10,0)--(11.5,0)--(11.5,1.5)--(10,1.5); 
  \draw[->] (8.3,0) -- (11.7,0) node[right] {$\bar \mu_l$};
  \draw[->] (10.0,-1.7) -- (10.0,1.7) node[above] {$\bar \nu_l$};  
  \node[below left] at (10,0){$0$};
\end{tikzpicture}}
\caption{Geometric visualizations of W$_\textup{CC}$-, C$_\textup{CC}$-, M$_\textup{CC}$-, and S$_\textup{CC}$-stationarity for 
	the program \eqref{eq:MPCC} w.r.t.\ an index $l\in I^{00}_\textup{CC}(\bar x)$.}
\label{fig:St_MPCC}
\end{figure}
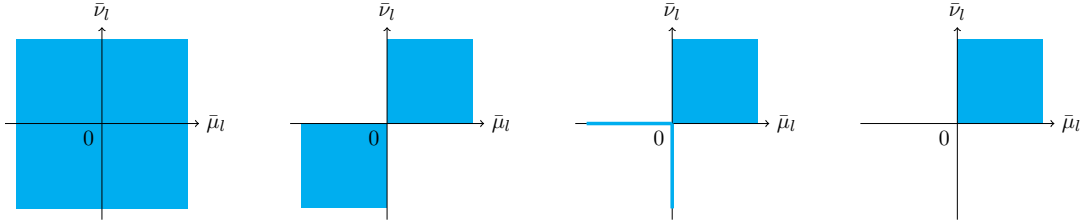

\section{Reformulation of or-constraints using NCP-functions}\label{sec:direct_treatment}

A continuous function $\varphi\colon\R^2\to\R$ which satisfies
\[
	\forall (a,b)\in\mathbb R^2\colon\quad
	\varphi(a,b)=0\,\Longleftrightarrow\,a,b\geq 0\,\land\,ab=0
\]
is referred to as NCP-function. 
By definition, NCP-functions can be used to reformulate complementarity systems as 
(possibly nonsmooth) equalities which is beneficial since the transformed system
can be tackled numerically with the aid of (semismooth) Newton or SQP methods, see e.g.\
\cite{Leyffer2006}.

Clearly, the zero level set of an NCP-function precisely equals the complementarity set
\begin{equation}\label{eq:def_compl_set}
	C:=\{(a,b)\in\R^2\,|\,a,b\geq 0\,\land\,ab=0\}.
\end{equation}
Defining the sets
\[
	A:=\{(a,b)\in\R^2\,|\,a>0\,\land\,b>0\},\qquad
	B:=\{(a,b)\in\R^2\,|\,a<0\,\lor\,b<0\},
\]
Bolzano's theorem yields that each NCP-function $\varphi\colon\R^2\to\R$ has 
precisely one of the following properties:
\begin{enumerate}[leftmargin=6em]
	\item[\textbf{NCP1}:] $\forall (a,b)\in A\cup B\colon\,\varphi(a,b)>0$,
	\item[\textbf{NCP2}:] $\forall (a,b)\in A\cup B\colon\,\varphi(a,b)<0$,
	\item[\textbf{NCP3}:] $\forall (a,b)\in A\colon\,\varphi(a,b)<0\,\land\,\forall (a,b)\in B\colon\,\varphi(a,b)>0$,
	\item[\textbf{NCP4}:] $\forall (a,b)\in A\colon\,\varphi(a,b)>0\,\land\,\forall (a,b)\in B\colon\,\varphi(a,b)<0$.
\end{enumerate}
This already has been mentioned in \cite[Corollary~1]{Galantai2012}.
Noting that $B\cup C=O$ holds for the set $O$ defined in \eqref{eq:def_O},
any NCP-function of type \textbf{NCP4} possesses the zero sublevel set $O$.
Particularly, for any such NCP-function $\varphi\colon\R^2\to\R$, we have
the relation
\[
	\forall (a,b)\in\R^2\colon\quad
	\varphi(a,b)\leq 0\,\Longleftrightarrow\,a\leq 0\,\lor\,b\leq 0\,\Longleftrightarrow\,(a,b)\in O.
\]
This observation yields the following definition.
\begin{definition}\label{def:NCP_OR}
	An NCP-function $\varphi\colon\R^2\to\R$ is said to be \emph{or-compatible} if it
	possesses property \textup{\textbf{NCP4}}.
\end{definition}

Clearly, if $\varphi\colon\R^2\to\R$ is an NCP-function possessing property \textbf{NCP3}, then
$-\varphi$ is an or-compatible NCP-function.
Below, we list three popular NCP-functions which are or-compatible while noting that there exist many
more examples:
\begin{itemize}
	\item the minimum function $\varphi_\textup{min}\colon\R^2\to\R$ given by
		\[
			\forall (a,b)\in\R^2\colon\quad
			\varphi_\textup{min}(a,b):=\min\{a;b\},
		\]
	\item the Fischer--Burmeister-type function $\varphi_\textup{FB}\colon\R^2\to\R$ defined via
		\[
			\forall (a,b)\in\R^2\colon\quad
			\varphi_\textup{FB}(a,b):=a+b-\sqrt{a^2+b^2}
		\]
		which dates back to \cite{Fischer1992}, and
	\item the Kanzow--Schwartz function $\varphi_\textup{KS}\colon\R^2\to\R$ from \cite{KanzowSchwartz2013}
		which is defined by
		\[
			\forall (a,b)\in\R^2\colon\quad
			\varphi_\textup{KS}(a,b):=
				\begin{cases}
					ab	&\text{if }a+b\geq 0,\\
					-\tfrac12(a^2+b^2)	&\text{if }a+b<0.
				\end{cases}
		\]
\end{itemize}
Obviously, $\varphi_\textup{min}$ is nonsmooth at all points from $\{(a,a)\in\R^2\,|\,a\in\R\}$ while
$\varphi_\textup{FB}$ is nonsmooth only at the origin. By construction, the function $\varphi_\textup{KS}$ 
is continuously differentiable, see \cite[Lemma~3.1]{KanzowSchwartz2013}, which makes it rather attractive 
in comparison to other NCP-functions.

For an arbitrary or-compatible NCP-function $\varphi\colon\R^2\to\R$, we now consider the surrogate 
\begin{equation}\label{eq:MPOC_NCP}\tag{MPOC$(\varphi)$}
	\begin{aligned}
		f(x)&\,\rightarrow\,\min&&&\\
		g_i(x)&\,\leq\,0&\qquad&i\in\mathcal M&\\
		h_j(x)&\,=\,0&\qquad&j\in\mathcal P&\\
		\varphi(G_l(x),H_l(x))&\,\leq\,0&\qquad&l\in\mathcal Q&
	\end{aligned}
\end{equation}
which is equivalent to \eqref{eq:MPOC}. Let $\bar x\in\R^n$ be an arbitrary feasible point of
\eqref{eq:MPOC} and, thus, of \eqref{eq:MPOC_NCP}. We set
\[
	I^\varphi(\bar x):=\{l\in\mathcal Q\,|\,\varphi(G_l(\bar x),H_l(\bar x))=0\}.
\]
Since $\varphi$ is or-compatible, we have $I^\varphi(\bar x)=\mathcal I(\bar x)$
for the set $\mathcal I(\bar x)$ defined in \eqref{eq:def_I}.
We now study the relationship between programs \eqref{eq:MPOC} and \eqref{eq:MPOC_NCP}
w.r.t.\ stationary points. 

Noting that \hyperref[eq:MPOC_NCP]{\textup{MPOC}$(\varphi_\textup{KS})$}
is a smooth program, we first investigate this particular model.
\begin{proposition}\label{prop:MPOC_KS}
	A feasible point $\bar x\in X$ of \eqref{eq:MPOC} is S-stationary if and only if it
	is a KKT point of \hyperref[eq:MPOC_NCP]{\textup{MPOC}$(\varphi_\textup{KS})$}.
\end{proposition}
\begin{proof}
$[\Longrightarrow]$ Let $\bar x$ be an S-stationary point of \eqref{eq:MPOC}.
Then, we find multipliers $\lambda_i\geq 0$ ($i\in I^g(\bar x)$), $\rho_j$ ($j\in\mathcal P$),
$\mu_l\geq 0$ ($l\in I^{0+}(\bar x)$), and $\nu_l\geq 0$ ($l\in I^{+0}(\bar x)$) such that
\begin{align*}
	0
	&
		=\nabla f(\bar x)
		+\sum\limits_{i\in I^g(\bar x)}\lambda_i\nabla g_i(\bar x)
		+\sum\limits_{j\in\mathcal P}\rho_j\nabla h_j(\bar x)\\
	&\qquad
		+\sum\limits_{l\in I^{0+}(\bar x)}\mu_l\nabla G_l(\bar x)
		+\sum\limits_{l\in I^{+0}(\bar x)}\nu_l\nabla H_l(\bar x)
\end{align*}
holds. Next, we set
\[
	\forall l\in I^{\varphi_\textup{KS}}(\bar x)\colon\quad
		\xi_l:= \begin{cases}
					\mu_l/H_l(\bar x)	&\text{if }l\in I^{0+}(\bar x),\\
					\nu_l/G_l(\bar x)	&\text{if }l\in I^{+0}(\bar x),\\
					0					&\text{if }l\in I^{00}(\bar x),
				\end{cases}
\]
which allows us to rewrite the above equation as
\begin{align*}
	0&
		=\nabla f(\bar x)
		+\sum\limits_{i\in I^g(\bar x)}\lambda_i\nabla g_i(\bar x)
		+\sum\limits_{j\in\mathcal P}\rho_j\nabla h_j(\bar x)\\
	 &\qquad 
	 	+\sum\limits_{l\in I^{\varphi_\textup{KS}}(\bar x)}
			\xi_l\bigl(H_l(\bar x)\nabla G_l(\bar x)+G_l(\bar x)\nabla H_l(\bar x)\bigr).
\end{align*}
By definition of $\varphi_\textup{KS}$ and $\xi_l\geq 0$ for all $l\in I^{\varphi_\textup{KS}}(\bar x)$,
$\bar x$ is a KKT point of \hyperref[eq:MPOC_NCP]{\textup{MPOC}$(\varphi_\textup{KS})$}.\\
$[\Longleftarrow]$ If $\bar x$ is a KKT point of \hyperref[eq:MPOC_NCP]{\textup{MPOC}$(\varphi_\textup{KS})$},
	we find multipliers $\bar\lambda_i\geq 0$ ($i\in I^g(\bar x)$), $\bar\rho_j$ ($j\in\mathcal P$), and
	$\bar\xi_l\geq 0$ ($l\in I^{\varphi_\textup{KS}}(\bar x)$) such that
	\begin{align*}
	0&
		=\nabla f(\bar x)
		+\sum\limits_{i\in I^g(\bar x)}\bar \lambda_i\nabla g_i(\bar x)
		+\sum\limits_{j\in\mathcal P}\bar \rho_j\nabla h_j(\bar x)\\
	 &\qquad 
	 	+\sum\limits_{l\in I^{\varphi_\textup{KS}}(\bar x)}
			\bar \xi_l\bigl(H_l(\bar x)\nabla G_l(\bar x)+G_l(\bar x)\nabla H_l(\bar x)\bigr)
\end{align*}
	is valid. Now, we define $\bar\mu_l:=\bar\xi_l H_l(\bar x)$ ($l\in I^{0+}(\bar x)\cup I^{00}(\bar x)$)
	as well as $\bar\nu_l:=\bar\xi_l G_l(\bar x)$ ($l\in I^{+0}(\bar x)\cup I^{00}(\bar x)$) in order to see
	that $\bar x$ is S-stationary for \eqref{eq:MPOC}.
\end{proof}

The above result justifies to solve the smooth standard nonlinear problem 
\hyperref[eq:MPOC_NCP]{\textup{MPOC}$(\varphi_\textup{KS})$} instead of
the disjunctive program \eqref{eq:MPOC} in order to find S-stationary points of the latter.
However, it needs to be noted that \hyperref[eq:MPOC_NCP]{\textup{MPOC}$(\varphi_\textup{KS})$}
is still a challenging problem due to the combinatorial structure of its feasible set.
Additionally, if $I^{00}(\bar x)\neq\varnothing$ holds true for some feasible point
$\bar x\in\R^n$ of \eqref{eq:MPOC}, then for each $l\in I^{00}(\bar x)$, the gradient of
the map $x\mapsto\varphi_\textup{KS}(G_l(x),H_l(x))$ vanishes at $\bar x$. This particularly means
that popular constraint qualifications like MFCQ or LICQ do not hold at $\bar x$
for \hyperref[eq:MPOC_NCP]{\textup{MPOC}$(\varphi_\textup{KS})$}.
However, it is possible to obtain the following result.
\begin{lemma}\label{lem:MPOC_LICQ_yields_GCQ_KS}
	Let $\bar x\in\R^n$ be a feasible point of \eqref{eq:MPOC} where MPOC-LICQ is valid.
	Then, GCQ holds for \hyperref[eq:MPOC_NCP]{\textup{MPOC}$(\varphi_\textup{KS})$}
	at $\bar x$.
\end{lemma}
\begin{proof}
	Let $\mathcal L_X^\textup{KS}(\bar x)$ be the linearization cone associated with
	program \hyperref[eq:MPOC_NCP]{\textup{MPOC}$(\varphi_\textup{KS})$} at $\bar x$. 
	By standard arguments, $\mathcal T_X(\bar x)\subset\mathcal L_X^\textup{KS}(\bar x)$
	holds true, and this yields the inclusion 
	$\mathcal L_X^\textup{KS}(\bar x)^\circ\subset\mathcal T_X(\bar x)^\circ$.
	In order to verify that GCQ holds for \hyperref[eq:MPOC_NCP]{\textup{MPOC}$(\varphi_\textup{KS})$} 
	at $\bar x$, we only need to show the opposite inclusion.
	
	Let $I\subset I^{00}(\bar x)$ be arbitrarily chosen.
	We consider the program
	\begin{equation}\label{eq:MPOC_I}\tag{MPOC$(\bar x,I)$}
		\begin{aligned}
			f(x)&\,\to\,\min&&&\\
			g_i(x)&\,\leq\,0&\qquad&i\in\mathcal M\\
			h_j(x)&\,=\,0&\qquad&j\in\mathcal P\\
			G_l(x)&\,\leq\,0&\qquad&l\in I^{0+}(\bar x)\cup I\\
			H_l(x)&\,\leq\,0&\qquad&l\in I^{+0}(\bar x)\cup(I^{00}(\bar x)\setminus I)
		\end{aligned}
	\end{equation}
	whose feasible set will be denoted by $X(\bar x,I)$ in the subsequent considerations. 
	Locally around $\bar x$,
	the family $\{X(\bar x,I)\}_{I\subset I^{00}(\bar x)}$ provides a decomposition of $X$
	which is why the relation
	\[
		\mathcal T_X(\bar x)=\bigcup\limits_{I\subset I^{00}(\bar x)}\mathcal T_{X(\bar x,I)}(\bar x)
	\]
	holds, cf.\ \cite[Lemma~3.1]{FlegelKanzow2005} or 
	\cite[Lemma~5.1]{Mehlitz2019} for related results associated with MPCCs or MPSCs, respectively.
	Noting that the validity of MPOC-LICQ implies that LICQ holds for \eqref{eq:MPOC_I}
	at $\bar x$ for each $I\subset I^{00}(\bar x)$, we can infer
	\[
		\mathcal T_X(\bar x)=\bigcup\limits_{I\subset I^{00}(\bar x)}\mathcal L_{X(\bar x,I)}(\bar x),
	\]
	and polarizing this formula shows
	\[
		\mathcal T_X(\bar x)^\circ=\bigcap\limits_{I\subset I^{00}(\bar x)}\mathcal L_{X(\bar x,I)}(\bar x)^\circ
	\]
	where $\mathcal L_{X(\bar x,I)}(\bar x)$ denotes the linearization cone of \eqref{eq:MPOC_I} at $\bar x$.
	
	Pick $\eta\in\mathcal T_X(\bar x)^\circ$ arbitrarily.
	The above considerations show 
	\[
		\eta\in\mathcal L_{X(\bar x,\varnothing)}(\bar x)^\circ\cap\mathcal L_{X(\bar x,I^{00}(\bar x))}(\bar x)^\circ.
	\]
	Hence, we find multipliers $\lambda^s_i\geq 0$ ($s=1,2$ and $i\in I^g(\bar x)$),
	$\rho^s_j$ ($s=1,2$ and $j\in\mathcal P$), $\mu^s_l\geq 0$ ($s=1,2$ and $l\in I^{0+}(\bar x)$), $\nu^s_l\geq 0$ 
	($s=1,2$ and $l\in I^{+0}(\bar x)$), $\mu^2_l\geq 0$ ($l\in I^{00}(\bar x)$), and $\nu^1_l\geq 0$ ($l\in I^{00}(\bar x)$)
	such that $\eta$ possesses the following representations
	\begin{align*}
		\eta
		&=
		\sum\limits_{i\in I^g(\bar x)}\lambda^1_i\nabla g_i(\bar x)
		+\sum\limits_{j\in\mathcal P}\rho^1_j\nabla h_j(\bar x)\\
		&\qquad
		+\sum\limits_{l\in I^{0+}(\bar x)}\mu^1_l\nabla G_l(\bar x)
		+\sum\limits_{l\in I^{+0}(\bar x)\cup I^{00}(\bar x)}\nu^1_l\nabla H_l(\bar x)\\
		&=
		\sum\limits_{i\in I^g(\bar x)}\lambda^2_i\nabla g_i(\bar x)
		+\sum\limits_{j\in\mathcal P}\rho^2_j\nabla h_j(\bar x)\\
		&\qquad
		+\sum\limits_{l\in I^{0+}(\bar x)\cup I^{00}(\bar x)}\mu^2_l\nabla G_l(\bar x)
		+\sum\limits_{l\in I^{+0}(\bar x)}\nu^2_l\nabla H_l(\bar x).
	\end{align*}
	This particularly shows
	\begin{align*}
		0
		&=
		\sum\limits_{i\in I^g(\bar x)}(\lambda^1_i-\lambda^2_i)\nabla g_i(\bar x)
		+\sum\limits_{j\in\mathcal P}(\rho^1_j-\rho^2_j)\nabla h_j(\bar x)\\
		&\qquad
		+\sum\limits_{l\in I^{0+}(\bar x)}(\mu^1_l-\mu^2_l)\nabla G_l(\bar x)
		+\sum\limits_{l\in I^{+0}(\bar x)}(\nu^1_l-\nu^2_l)\nabla H_l(\bar x)\\
		&\qquad
		-\sum\limits_{l\in I^{00}(\bar x)}\mu^2_l\nabla G_l(\bar x)
		+\sum\limits_{l\in I^{00}(\bar x)}\nu^1_l\nabla H_l(\bar x).
	\end{align*}
	Exploiting the validity of MPOC-LICQ, this leads to $\mu^2_l=\nu^1_l=0$ for all
	$l\in I^{00}(\bar x)$, i.e.\
	\begin{align*}
		\eta
		&=
		\sum\limits_{i\in I^g(\bar x)}\lambda^1_i\nabla g_i(\bar x)
		+\sum\limits_{j\in\mathcal P}\rho^1_j\nabla h_j(\bar x)\\
		&\qquad
		+\sum\limits_{l\in I^{0+}(\bar x)}\mu^1_l\nabla G_l(\bar x)
		+\sum\limits_{l\in I^{+0}(\bar x)}\nu^1_l\nabla H_l(\bar x)
	\end{align*}
	holds true. A simple calculation shows that this means 
	$\eta\in \mathcal L^\textup{KS}_X(\bar x)^\circ$.
\end{proof}

In contrast to \hyperref[eq:MPOC_NCP]{\textup{MPOC}$(\varphi_\textup{KS})$}, the programs
\hyperref[eq:MPOC_NCP]{\textup{MPOC}$(\varphi_\textup{min})$} and 
\hyperref[eq:MPOC_NCP]{\textup{MPOC}$(\varphi_\textup{FB})$} are nonsmooth. 
However, using suitable subdifferential constructions, it is possible to state
KKT-type systems associated with these optimization problems as well.
Noting that the active set $I^\varphi(\bar x)$ is directly related to the complementarity
set $C$ defined in \eqref{eq:def_compl_set}, only subdifferential information of 
$\varphi$ on $C$ is relevant for the 
characterization of the associated KKT systems.
Using Clarke's constructions, see \cite{Clarke1983}, we obtain
\begin{align*}
		\partial^\textup{C}\varphi_\textup{min}(a,b)
		&=
			\begin{cases}
				\{\mathtt e^2_1\}
					&\text{if } a=0\,\land\, b>0,\\
				\{\mathtt e^2_2\}
					&\text{if }a>0\,\land\,b=0,\\
				\conv\left\{\mathtt e^2_1,\mathtt e^2_2\right\}
					&\text{if }a=b=0,
			\end{cases}\\
		\partial^\textup{C}\varphi_\textup{FB}(a,b)
		&=
			\begin{cases}
				\{\mathtt e^2_1\}
					&\text{if } a=0\,\land\, b>0,\\
				\{\mathtt e^2_2\}
					&\text{if }a>0\,\land\,b=0,\\
				\mathbb B^1(\mathtt e^2)
					&\text{if }a=b=0
			\end{cases}
\end{align*}
for all $(a,b)\in C$.  Sharper results can be obtained with Mordukhovich's subdifferential,
see \cite{Mordukhovich2006}, which computes as
\begin{align*}
		\partial^\textup{M}\varphi_\textup{min}(a,b)
		&=
			\begin{cases}
				\{\mathtt e^2_1\}
					&\text{if } a=0\,\land\, b>0,\\
				\{\mathtt e^2_2\}
					&\text{if }a>0\,\land\,b=0,\\
				\left\{\mathtt e^2_1,\mathtt e^2_2\right\}
					&\text{if }a=b=0,
			\end{cases}\\
		\partial^\textup{M}\varphi_\textup{FB}(a,b)
		&=
			\begin{cases}
				\{\mathtt e^2_1\}
					&\text{if } a=0\,\land\, b>0,\\
				\{\mathtt e^2_2\}
					&\text{if }a>0\,\land\,b=0,\\
				\mathbb S^1(\mathtt e^2)
					&\text{if }a=b=0.
			\end{cases}
\end{align*}
Using these formulas and suitable chain rules for the underlying subdifferentials, respective
KKT-type systems associated with \hyperref[eq:MPOC_NCP]{\textup{MPOC}$(\varphi_\textup{min})$} and 
\hyperref[eq:MPOC_NCP]{\textup{MPOC}$(\varphi_\textup{FB})$} can be derived, 
see \cite[Theorem~5.6.2]{Vinter2000} and \cite[Theorem~5.21]{Mordukhovich2006}, respectively.
For simplicity, we refer to these first-order systems as KKT systems again and specify
the underlying subdifferential construction.

The upcoming result which addresses \hyperref[eq:MPOC_NCP]{\textup{MPOC}$(\varphi_\textup{min})$}
can be validated exploiting a similar strategy as used for
the derivation of \cref{prop:MPOC_KS} doing some nearby changes.
That is why its proof is omitted here.
\begin{proposition}\label{prop:MPOC_min}
\begin{enumerate}
	\item A feasible point $\bar x\in\R^n$ of \eqref{eq:MPOC} is W-stationary if and only
	if it is a KKT point of \hyperref[eq:MPOC_NCP]{\textup{MPOC}$(\varphi_\textup{min})$}
	w.r.t.\ Clarke's subdifferential.
	\item A feasible point $\bar x\in\R^n$ of \eqref{eq:MPOC} is M-stationary if and only
	if it is a KKT point of \hyperref[eq:MPOC_NCP]{\textup{MPOC}$(\varphi_\textup{min})$}
	w.r.t.\ Mordukhovich's subdifferential.
\end{enumerate}
\end{proposition}

Finally, we consider the KKT system of \hyperref[eq:MPOC_NCP]{\textup{MPOC}$(\varphi_\textup{FB})$}.
\begin{proposition}\label{prop:MPOC_FB}
	For a feasible point $\bar x\in\R^n$ of \eqref{eq:MPOC}, the following statements are equivalent:
	\begin{enumerate}
		\item[(a)] $\bar x$ is W-stationary,
		\item[(b)] $\bar x$ is a KKT point of \hyperref[eq:MPOC_NCP]{\textup{MPOC}$(\varphi_\textup{FB})$}
			w.r.t.\ Clarke's subdifferential, and
		\item[(c)] $\bar x$ is a KKT point of \hyperref[eq:MPOC_NCP]{\textup{MPOC}$(\varphi_\textup{FB})$}
			w.r.t.\ Mordukhovich's subdifferential.
	\end{enumerate}
\end{proposition}
\begin{proof}
	The implication (c)$\Longrightarrow$(b) is trivial due to 
	$\partial ^\textup{M}\varphi_\textup{FB}(a,b)\subset\partial^\textup C\varphi_\textup{FB}(a,b)$
	for all $(a,b)\in\R^2$. Furthermore, (b)$\Longrightarrow$(a) follows easily by the fact
	$\mathbb B^1(\mathtt e^2)\subset\R^2_+$.
	It remains to show (a)$\Longrightarrow$(c).
	
	Thus, let $\bar x$ be a W-stationary point of \eqref{eq:MPOC}.
	Then, we find multipliers $\lambda_i\geq 0$ ($i\in I^g(\bar x)$), $\rho_j$ ($j\in\mathcal P$),
	$\mu_l\geq 0$ ($l\in I^{0+}(\bar x)\cup I^{00}(\bar x)$), 
	and $\nu_l\geq 0$ ($l\in I^{+0}(\bar x)\cup I^{00}(\bar x)$) such that
	\eqref{eq:WSt} holds. Let us assume $I^{00}(\bar x)\neq\varnothing$ (otherwise, the proof is straightforward).
	Pick an index $l\in I^{00}(\bar x)$ and define $\xi_l:=\mu_l+\nu_l+\sqrt{2\mu_l\nu_l}\geq 0$.
	In case $\xi_l=0$, we set $\alpha_l=\beta_l:=1-\tfrac{\sqrt 2}{2}$.
	Otherwise, we define $\alpha_l:=\mu_l/\xi_l$ and $\beta_l:=\nu_l/\xi_l$. 
	By construction, the relation 
	$(\alpha_l,\beta_l)\in\mathbb S^1(\mathtt e^2)=\partial^\textup{M}\varphi_\textup{FB}(G_l(\bar x),H_l(\bar x))$
	follows. 
	Setting $\xi_l:=\mu_l$, $\alpha_l:=1$, and $\beta_l:=0$ for all $l\in I^{0+}(\bar x)$ 
	as well as $\xi_l:=\nu_l$, $\alpha_l:=0$, and $\beta_l:=1$ for all $l\in I^{+0}(\bar x)$, we have
	\begin{align*}
	&0	=\nabla f(\bar x)
		+\sum\limits_{i\in I^g(\bar x)}\lambda_i\nabla g_i(\bar x)
		+\sum\limits_{j\in\mathcal P}\rho_j\nabla h_j(\bar x)\\
	&\qquad
		+\sum\limits_{l\in I^{\varphi_{\textup{FB}}(\bar x)}}
			\xi_l\bigl(\alpha_l\nabla G_l(\bar x)+\beta_l\nabla H_l(\bar x)\bigr)\\
	&\forall l\in I^{\varphi_\textup{FB}}(\bar x)\colon\,\xi_l\geq 0,\\
	&\forall l\in I^{\varphi_\textup{FB}}(\bar x)\colon\,(\alpha_l,\beta_l)\in 
		\partial^\textup{M}\varphi_\textup{FB}(G_l(\bar x),H_l(\bar x)),
	\end{align*}
	i.e.\ $\bar x$ is a KKT point of \hyperref[eq:MPOC_NCP]{\textup{MPOC}$(\varphi_\textup{FB})$}
	w.r.t.\ Mordukhovich's subdifferential.
\end{proof}

Let us briefly point the reader's attention to the fact that the use of Mordukhovich's 
subdifferential construction w.r.t.\ the function $\varphi$ in the KKT system associated
with \eqref{eq:MPOC_NCP} does not automatically lead to the identification of M-stationary
points of \eqref{eq:MPOC} as \cref{prop:MPOC_FB} demonstrates.

The above \cref{prop:MPOC_KS,prop:MPOC_min,prop:MPOC_FB} suggest to solve \eqref{eq:MPOC_NCP} 
instead of \eqref{eq:MPOC} in order to identify stationary points of the latter. Noting
that at least \hyperref[eq:MPOC_NCP]{\textup{MPOC}$(\varphi_\textup{KS})$} is a smooth
problem, this can be done exploiting standard solvers from nonlinear programming. 
Suitable methods from nonsmooth optimization can be used to tackle
\hyperref[eq:MPOC_NCP]{\textup{MPOC}$(\varphi_\textup{min })$} and
\hyperref[eq:MPOC_NCP]{\textup{MPOC}$(\varphi_\textup{FB})$} numerically.

\section{Tranformation into other disjunctive programs}\label{sec:disjunctive_modification}

\subsection{Relations to switching-constrained programming}\label{sec:SC_ref}

Let us consider the switching-constrained optimization problem
\begin{equation}\label{eq:MPOC_SC}\tag{SC-MPOC}
	\begin{aligned}
		f(x)&\,\to\,\min\limits_{x,y,z}&&&\\
			g_i(x)&\,\leq\,0&\qquad&i\in\mathcal M&\\
			h_j(x)&\,=\,0&&j\in\mathcal P&\\
			y_l,z_l&\,\leq\,0&&l\in\mathcal Q&\\
			(G_l(x)-y_l)(H_l(x)-z_l)&\,=\,0&&l\in\mathcal Q&
	\end{aligned}
\end{equation}
associated with \eqref{eq:MPOC}. One can easily check that for each feasible point $\bar x\in X$
of \eqref{eq:MPOC}, we find $\bar y,\bar z\in\R^q$ such that $(\bar x,\bar y,\bar z)$ is feasible
to \eqref{eq:MPOC_SC}. On the contrary, if $(\tilde x,\tilde y,\tilde z)\in\R^n\times\R^q\times\R^q$
is feasible to \eqref{eq:MPOC_SC}, then $\tilde x$ is feasible to \eqref{eq:MPOC}.
This observation has been used in \cite[Section~7.1]{Mehlitz2019} in order to show that
\eqref{eq:MPOC} and \eqref{eq:MPOC_SC} are somehow equivalent w.r.t.\ global
minimizers while the local minimizers of \eqref{eq:MPOC} can be found among
the local minimizers of \eqref{eq:MPOC_SC}. Moreover, it has been shown that whenever
$(\tilde x,\tilde y,\tilde z)$ is a local minimizer of \eqref{eq:MPOC_SC} where 
$I^{0-}(\tilde x)\cup I^{-0}(\tilde x)\cup I^{00}(\tilde x)=\varnothing$ holds, then $\tilde x$
is a local minimizer of \eqref{eq:MPOC}. These results justify to consider the switching model
\eqref{eq:MPOC_SC} instead of \eqref{eq:MPOC}. 
However, one has to notice that this transformation comes for the price of $2q$ additional slack
variables and potential artificial local minimizers.

Let us compare \eqref{eq:MPOC} and \eqref{eq:MPOC_SC} w.r.t.\ stationary points
since local minimizers of \eqref{eq:MPOC} correspond to local minimizers of
\eqref{eq:MPOC_SC} which satisfy certain stationarity conditions under validity
of constraint qualifications.
It follows from 
\cite[Section~7.2]{Mehlitz2019} that the W-, M-, and S-stationary points
of \eqref{eq:MPOC} can be found among the W$_\textup{SC}$-, M$_\textup{SC}$-, and S$_\textup{SC}$-
stationary points of \eqref{eq:MPOC_SC}. As we will see below,
the converse statement is also true in certain situations.
\begin{proposition}\label{prop:St_sur_SC}
	Let $(\bar x,\bar y,\bar z)\in\R^n\times\R^q\times\R^q$ be feasible to \eqref{eq:MPOC_SC}
	and assume that the index sets $I^{0-}(\bar x)$ and $I^{-0}(\bar x)$ are empty.
	If $(\bar x,\bar y,\bar z)$ is W$_\textup{SC}$-stationary (M$_\textup{SC}$-stationary, S$_\textup{SC}$-stationary) 
	for \eqref{eq:MPOC_SC},
	then it is W-stationary	(M-stationary, S-stationary) for \eqref{eq:MPOC}.
\end{proposition}
\begin{proof}
	First, we set
	\begin{align*}
		I^{G}(\bar x,\bar y,\bar z)&:=\{l\in\mathcal Q\,|\,G_l(\bar x)=\bar y_l\,\land\,H_l(\bar x)\neq\bar z_l\},\\
		I^{H}(\bar x,\bar y,\bar z)&:=\{l\in\mathcal Q\,|\,G_l(\bar x)\neq\bar y_l\,\land\,H_l(\bar x)=\bar z_l\},\\
		I^{GH}(\bar x,\bar y,\bar z)&:=\{l\in\mathcal Q\,|\,G_l(\bar x)=\bar y_l\,\land\,H_l(\bar x)=\bar z_l\}.
	\end{align*}
	Let $(\bar x,\bar y,\bar z)$ be W$_\textup{SC}$-stationary for \eqref{eq:MPOC_SC}.
	Then, after elimination of the multipliers corresponding to the 
	inequality constraints on the variables $\bar y$ and $\bar z$,
	there are multipliers $\lambda_i\geq 0$ ($i\in I^g(\bar x)$), $\rho_j$ ($j\in\mathcal P$),
	$\mu_l\geq 0$ ($l\in I^G(\bar x,\bar y,\bar z)\cup I^{GH}(\bar x,\bar y,\bar z)$), and
	$\nu_l\geq 0$ ($l\in I^H(\bar x,\bar y,\bar z)\cup I^{GH}(\bar x,\bar y,\bar z)$) which satisfy
	\begin{equation}\label{eq:WSt_sur_SC}
		\begin{split}
		&0=\nabla f(\bar x)+\sum\limits_{i\in I^g(\bar x)}\lambda_i\nabla g_i(\bar x)
			+\sum\limits_{j\in\mathcal P}\rho_j\nabla h_j(\bar x)\\
			&\qquad
			+\sum\limits_{l\in I^G(\bar x,\bar y,\bar z)\cup I^{GH}(\bar x,\bar y,\bar z)}\mu_l\nabla G_l(\bar x)
			+\sum\limits_{l\in I^H(\bar x,\bar y,\bar z)\cup I^{GH}(\bar x,\bar y,\bar z)}\nu_l\nabla H_l(\bar x)\\
		&\forall l\in I^G(\bar x,\bar y,\bar z)\cup I^{GH}(\bar x,\bar y,\bar z)\colon\,\mu_l\bar y_l=0,\\
		&\forall l\in I^H(\bar x,\bar y,\bar z)\cup I^{GH}(\bar x,\bar y,\bar z)\colon\,\nu_l\bar z_l=0.
		\end{split}
	\end{equation}
	Obviously, we have
	\[
		\{l\in I^G(\bar x,\bar y,\bar z)\cup I^{GH}(\bar x,\bar y,\bar z)\,|\,\bar y_l=0\}
		=
		\{l\in\mathcal Q\,|\,G_l(\bar x)=\bar y_l=0\}
		=
		I^{0+}(\bar x)\cup I^{00}(\bar x)
	\]
	and
	\[
		\{l\in I^H(\bar x,\bar y,\bar z)\cup I^{GH}(\bar x,\bar y,\bar z)\,|\,\bar z_l=0\}
		=
		\{l\in\mathcal Q\,|\,H_l(\bar x)=\bar z_l=0\}
		=
		I^{+0}(\bar x)\cup I^{00}(\bar x)
	\]
	from $I^{0-}(\bar x)=I^{-0}(\bar x)=\varnothing$. Thus, the multiplier $\mu_l$ can be positive only for
	indices $l\in I^{0+}(\bar x)\cup I^{00}(\bar x)$ while
	$\nu_l$ can be positive only for $l\in I^{+0}(\bar x)\cup I^{00}(\bar x)$.	
	Consequently, \eqref{eq:WSt_sur_SC} shows that $\bar x$ is W-stationary for \eqref{eq:MPOC}.
	
	Next, we suppose that $(\bar x,\bar y,\bar z)$ is M$_\textup{SC}$-stationary for \eqref{eq:MPOC_SC}.
	Then, the above multipliers additionally need to satisfy
	\[
		\forall l\in I^{GH}(\bar x,\bar y,\bar z)\colon\quad \mu_l\nu_l=0.
	\]
	Clearly, the assumption $I^{0-}(\bar x)=I^{-0}(\bar x)=\varnothing$ yields
	\begin{equation}\label{eq:biactive_set_MPOC_SC}
		I^{GH}(\bar x,\bar y,\bar z)
		=
		\{l\in\mathcal Q\,|\,G_l(\bar x)=\bar y_l\,\land\,H_l(\bar x)=\bar z_l\}
		=
		I^{00}(\bar x)\cup I^{--}(\bar x).
	\end{equation}
	Thus, the above considerations lead to $\mu_l\nu_l=0$ for all $l\in I^{00}(\bar x)$,
	i.e.\ $\bar x$ is M-stationary for \eqref{eq:MPOC}.
	
	Finally, suppose that $(\bar x,\bar y,\bar z)$ is S$_\textup{SC}$-stationary for \eqref{eq:MPOC_SC}.
	In this case, the above multipliers additionally satisfy the condition
	\[
		\forall l\in I^{GH}(\bar x,\bar y,\bar z)\colon\quad \mu_l=0\,\land\,\nu_l=0.
	\]
	Then, \eqref{eq:biactive_set_MPOC_SC} yields that $\mu_l=0$ and $\nu_l=0$ hold for
	all $l\in I^{00}(\bar x)$ which means that $\bar x$ is already S-stationary for
	\eqref{eq:MPOC}.	
\end{proof}

Let us visualize the assertion of \cref{prop:St_sur_SC} by means of the following toy program
taken from \cite{Mehlitz2019}.
\begin{example}\label{ex:additional_stationary_points}
	Consider the simple or-constrained program
		\begin{equation}\label{eq:simple_or_constrained_program}
			\begin{split}
				(x_1-1)^2&\,\to\,\min\\
				x_1\,\leq\,0\,\lor\,x_2&\,\leq\,0.
			\end{split}
		\end{equation}
	The set of its global minimizers is given by $G:=\{(1,x_2)\,|\,x_2\leq 0\}$ while there
	are additional local minimizers at all points from $L:=\{(0,x_2)\,|\,x_2>0\}$. One can
	easily check that all points from $G$ are S-stationary for \eqref{eq:simple_or_constrained_program}
	while the points from $L$ are only M-stationary. Note that there is an additional
	M-stationary point at $(0,0)$ which is not a local minimizer of 
	\eqref{eq:simple_or_constrained_program}.
	
	Now, we consider the switching-constrained surrogate problem 
		\begin{equation}
			\label{eq:simple_or_constrained_program_SC}
			\begin{split}
				(x_1-1)^2&\,\to\,\min\limits_{x,y,z}\\
				y,z&\,\leq\,0\\
				(x_1-y)(x_2-z)&\,=\,0
			\end{split}
		\end{equation}
	associated with \eqref{eq:simple_or_constrained_program}. By construction, all local
	minimizers and stationary points of \eqref{eq:simple_or_constrained_program} can be found
	among the local minimizers and stationary points of \eqref{eq:simple_or_constrained_program_SC}.	
	It has been mentioned in \cite[Example~7.1]{Mehlitz2019} that 
	\eqref{eq:simple_or_constrained_program_SC} possesses local minimizers whose $x$-components do
	not correspond to local minimizers of \eqref{eq:simple_or_constrained_program} e.g.\ at the points
	$(\bar x,\bar y,\bar z):=(0,0,0,-2)$ and $(\tilde x,\tilde y,\tilde z):=(0,-1,0,-2)$.
	Due to \cite[Theorem~7.2]{Mehlitz2019}, these points are M$_\textup{SC}$-stationary for
	\eqref{eq:simple_or_constrained_program_SC} since the latter is a switching-constrained
	program whose feasible region is defined via affine data functions only. 
	As mentioned above, $\bar x$ is an M-stationary point
	of \eqref{eq:simple_or_constrained_program} while $\tilde x$ is not. 
	Finally, observe that $I^{00}(\bar x)=\{1\}$ and $I^{0-}(\tilde x)=\{1\}$ hold.	
\end{example}

Due to the facts discussed above, it is reasonable to focus on the
computation of stationary points of \eqref{eq:MPOC_SC} 
in order to solve \eqref{eq:MPOC}. However, one has to keep in mind that
there are stationary solutions of \eqref{eq:MPOC_SC} that are not 
stationary for \eqref{eq:MPOC}, see \cref{prop:St_sur_SC} and \cref{ex:additional_stationary_points}.

In \cite{KanzowMehlitzSteck2019}, the authors suggest to modify relaxation
techniques for the numerical handling of MPCCs in order to tackle
switching-constrained optimization problems. The presented computational results
depict that adapted global relaxation schemes due to Scholtes, see \cite{Scholtes2001},
as well as Kanzow and Schwartz, see \cite{KanzowSchwartz2013}, are suitable for
that purpose. The adapted method due to Scholtes turned out to be the more robust
one which is why we briefly comment on this approach below. For details, we 
refer the interested reader to \cite{KanzowMehlitzSteck2019}.

For some parameter $t\geq 0$, let us investigate the relaxed nonlinear program
\begin{equation}\label{eq:MPOC_SC_Scholtes}\tag{SC-MPOC$_\textup{S}(t)$}
	\begin{aligned}
		f(x)&\,\to\,\min\limits_{x,y,z}&&&\\
			g_i(x)&\,\leq\,0&\qquad&i\in\mathcal M&\\
			h_j(x)&\,=\,0&&j\in\mathcal P&\\
			y_l,z_l&\,\leq\,0&&l\in\mathcal Q&\\
			-t\,\leq\,(G_l(x)-y_l)(H_l(x)-z_l)&\,\leq\,t&&l\in\mathcal Q.&
	\end{aligned}
\end{equation}
It possesses $m+4q$ inequality and $p$ equality constraints.
Clearly, for positive $t$, the feasible set of \eqref{eq:MPOC_SC_Scholtes} is
a superset of the feasible set associated with \eqref{eq:MPOC_SC}.
Moreover, the family of feasible sets associated with \eqref{eq:MPOC_SC_Scholtes} is
nested w.r.t.\ $t$ in such a way that for $t=0$, the feasible set of \eqref{eq:MPOC_SC}
is restored. Thus, for the numerical solution of \eqref{eq:MPOC_SC}, one
can choose a sequence $\{t_k\}_{k\in\N}$ of positive relaxation parameters converging
to zero and solve the associated relaxed nonlinear problems 
\hyperref[eq:MPOC_SC_Scholtes]{\textup{SC-MPOC}$_\textup{S}(t_k)$} using standard solvers
from nonlinear programming. Supposing that the computed sequence converges, its
limit point is feasible to \eqref{eq:MPOC_SC}. Furthermore, suitable assumptions 
can be imposed to guarantee that this limit point is at least W$_\textup{SC}$-stationary,
i.e.\ this approach is likely to produce W-stationary points of \eqref{eq:MPOC},
see \cite[Theorem~3.2]{KanzowMehlitzSteck2019}.

\subsection{Relations to complementarity-constrained programming}\label{sec:ref_CC}

	Let as consider the complementarity-constrained optimization problem
	\begin{equation}\label{eq:MPOC_CC}\tag{CC-MPOC}
	\begin{aligned}
		f(x)&\,\to\,\min\limits_{x,y,z}&&&\\
			g_i(x)&\,\leq\,0&\qquad&i\in\mathcal M&\\
			h_j(x)&\,=\,0&&j\in\mathcal P&\\
			G_l(x)-y_l&\,\leq\,0&&l\in\mathcal Q&\\
			H_l(x)-z_l&\,\leq\,0&&l\in\mathcal Q&\\
			0\,\leq\,y_l\,\perp\,z_l&\,\geq\,0&&l\in\mathcal Q&
	\end{aligned}
	\end{equation}
	associated with \eqref{eq:MPOC}. 
	Fix an arbitrary feasible point $\bar x\in X$ of \eqref{eq:MPOC} and define
	$\bar y,\bar z\in\R^q$ as stated below:
	\begin{equation}\label{eq:feasibility_MPOC_CC}
		\begin{split}
			\forall l\in\mathcal Q\colon\quad
			\bar y_l&:=
				\begin{cases}
					0		&l\in I^{-0}(\bar x)\cup I^{-+}(\bar x)\cup I^{0+}(\bar x)\cup I^{--}(\bar x)\cup I^{00}(\bar x),\\
					1		&l\in I^{0-}(\bar x),\\
					2G_l(\bar x)	&l\in I^{+-}(\bar x)\cup I^{+0}(\bar x),
				\end{cases}\\
			\bar z_l&:=
				\begin{cases}
					0		&l\in I^{0-}(\bar x)\cup I^{+-}(\bar x)\cup I^{+0}(\bar x)\cup I^{--}(\bar x)\cup I^{00}(\bar x),\\
					1		&l\in I^{-0}(\bar x),\\
					2H_l(\bar x)	&l\in I^{-+}(\bar x)\cup I^{0+}(\bar x).
				\end{cases}
		\end{split}
	\end{equation}
	Clearly, $(\bar x,\bar y,\bar z)$ is feasible to \eqref{eq:MPOC_CC}.
	On the other hand, one can easily check that for each feasible point 
	$(\tilde x,\tilde y,\tilde z)\in\R^n\times\R^q\times\R^q$ of \eqref{eq:MPOC_CC},
	$\tilde x$ is feasible to \eqref{eq:MPOC}.
	
	Based on this observation, we obtain the following result by standard arguments.
	\begin{proposition}\label{prop:relationship_CC}
		\begin{enumerate}
			\item Let $\bar x\in X$ be a locally (globally) optimal solution of \eqref{eq:MPOC}.
				Furthermore, let $\bar y,\bar z\in\R^q$ be the vectors defined in \eqref{eq:feasibility_MPOC_CC}.
				Then, $(\bar x,\bar y,\bar z)$ is a locally (globally) optimal solution of \eqref{eq:MPOC_CC}.
			\item Let $(\tilde x,\tilde y,\tilde z)\in\R^n\times\R^q\times\R^q$ be a globally optimal
				solution of \eqref{eq:MPOC_CC}. Then, $\tilde x$ is a globally optimal solution of \eqref{eq:MPOC}.
		\end{enumerate}
	\end{proposition}
	
	The upcoming example
	shows that the second statement of \cref{prop:relationship_CC} cannot
	be extended to local minimizers. This observation parallels the one for the switching-constrained
	reformulation of \eqref{eq:MPOC} discussed in \cref{sec:SC_ref}.  
	\begin{example}\label{ex:local_relationship}
		Let us consider \eqref{eq:simple_or_constrained_program}
		as well as its complementarity-constrained reformulation
		\begin{equation}\label{eq:simple_CC_ref}
			\begin{split}
				(x_1-1)^2&\,\to\,\min\limits_{x,y,z}\\
				x_1-y&\,\leq\,0\\
				x_2-z&\,\leq\,0\\
				0\,\leq y\,\perp\,z&\,\geq\,0.
			\end{split}
		\end{equation}
		Using similar arguments as in \cite[Example~7.1]{Mehlitz2019}, one can check that the
		points $(0,-1,0,1)$ and $(0,0,0,1)$ are local minimizers of \eqref{eq:simple_CC_ref} which do 
		not correspond to the minimizers of \eqref{eq:simple_or_constrained_program} 
		characterized in \cref{ex:additional_stationary_points}.
	\end{example}
	
	Similar to \cite[Lemma~7.2]{Mehlitz2019}, we obtain the following result.
	\begin{proposition}\label{prop:local_relationship_CC}
		Let $(\bar x,\bar y,\bar z)\in\R^n\times\R^q\times\R^q$ be a locally optimal solution of \eqref{eq:MPOC_CC}
		and assume that the index sets $I^{-0}(\bar x)$, $I^{0-}(\bar x)$, and $I^{00}(\bar x)$ are empty.
		Then, $\bar x$ is a local minimizer of \eqref{eq:MPOC}.
	\end{proposition}
	
	Summarizing the above facts, the transformation \eqref{eq:MPOC_CC} comes for the price of $2q$
	slack variables and potential additional local minimizers. These are precisely those disadvantages we
	had to face when using the switching-constrained surrogate \eqref{eq:MPOC_SC}.
	
	Finally, we want to compare \eqref{eq:MPOC} and \eqref{eq:MPOC_CC} w.r.t.\ stationary points.
	The upcoming result shows that we can find the W-, M-, and S-stationary points of \eqref{eq:MPOC}
	among the C$_\textup{CC}$-, M$_\textup{CC}$-, and S$_\textup{CC}$-stationary points of \eqref{eq:MPOC_CC}.
	\begin{proposition}\label{prop:St_MPOC_to_St_CC}
		Let $\bar x\in X$ be a W-stationary (M-stationary, S-stationary) point of \eqref{eq:MPOC}.
		Furthermore, let $\bar y,\bar z\in\R^q$ be the vectors defined in \eqref{eq:feasibility_MPOC_CC}.
		Then, $(\bar x,\bar y,\bar z)$ is C$_\textup{CC}$-stationary (M$_\textup{CC}$-stationary, 
		S$_\textup{CC}$-stationary) for \eqref{eq:MPOC_CC}.
	\end{proposition}
	\begin{proof}
		By definition of $\bar y$ and $\bar z$, we obtain
		\begin{align*}
			I^{0+}_\textup{CC}(\bar x,\bar y,\bar z)&=I^{-0}(\bar x)\cup I^{-+}(\bar x)\cup I^{0+}(\bar x),\\
			I^{+0}_\textup{CC}(\bar x,\bar y,\bar z)&=I^{0-}(\bar x)\cup I^{+-}(\bar x)\cup I^{+0}(\bar x),\\
			I^{00}_\textup{CC}(\bar x,\bar y,\bar z)&=I^{--}(\bar x)\cup I^{00}(\bar x).
		\end{align*}
		Since $\bar x$ is W-stationary for \eqref{eq:MPOC}, we find multipliers $\lambda_i\geq 0$ ($i\in I^g(\bar x)$),
		$\rho_j$ ($j\in\mathcal P$), $\mu_l\geq 0$ ($l\in I^{0+}(\bar x)\cup I^{00}(\bar x)$), and
		$\nu_l\geq 0$ ($l\in I^{+0}(\bar x)\cup I^{00}(\bar x)$) which satisfy \eqref{eq:WSt}.
		Now set $\mu_l:=0$ for all $l\in\mathcal Q\setminus(I^{0+}(\bar x)\cup I^{00}(\bar x))$ as well as
		$\nu_l:=0$ for all $l\in\mathcal Q\setminus(I^{+0}(\bar x)\cup I^{00}(\bar x)))$. 
		Furthermore, fix $\bar\mu:=-\mu$ and $\bar\nu:=-\nu$.  
		Then, these multipliers solve 
		\begin{align*}
		&0
		=\nabla f(\bar x)
		+\sum\limits_{i\in I^g(\bar x)}\lambda_i\nabla g_i(\bar x)
		+\sum\limits_{j\in\mathcal P}\rho_j\nabla h_j(\bar x)\\
		&\qquad
		+\sum\limits_{l\in\mathcal Q}\mu_l\nabla G_l(\bar x)
		+\sum\limits_{l\in\mathcal Q}\nu_l\nabla H_l(\bar x),\\
		&0=-\mu-\bar \mu,\qquad 0=-\nu-\bar \nu,\\
		&\forall i\in I^g(\bar x)\colon\,\lambda_i\geq 0,\\
		&0=\mu\cdot (G(\bar x)-\bar y),\qquad 0=\nu\cdot (H(\bar x)-\bar z),\\
		&\forall l\in I^{+0}_\textup{CC}(\bar x,\bar y,\bar z)\colon\,\bar\mu_l=0,\\
		&\forall l\in I^{0+}_\textup{CC}(\bar x,\bar y,\bar z)\colon\,\bar\nu_l=0,\\
		&\forall l\in I^{00}_\textup{CC}(\bar x,\bar y,\bar z)\colon\,\bar\mu_l\bar\nu_l\geq 0
		\end{align*}
		which is the C$_\textup{CC}$-stationarity system of \eqref{eq:MPOC_CC} at $(\bar x,\bar y,\bar z)$, i.e.\
		the latter point is C$_\textup{CC}$-stationary for \eqref{eq:MPOC_CC}.
		
		If $\bar x$ is M-stationary (S-stationary) for \eqref{eq:MPOC}, then the multipliers 
		$\mu$ and $\nu$ from above additionally satisfy $\mu_l\nu_l=0$ ($\mu_l=0$ and $\nu_l=0$)
		for all $l\in I^{00}(\bar x)$. This means that the new multipliers $\bar\mu$ and $\bar\nu$
		particularly satisfy $\bar\mu_l\bar\nu_l=0$ ($\bar\mu_l= 0$ and $\bar\nu_l= 0$) for all
		$l\in I^{00}_\textup{CC}(\bar x,\bar y,\bar z)$ which implies that $(\bar x,\bar y,\bar z)$
		is M$_\textup{CC}$-stationary (S$_\textup{CC}$-stationary) for \eqref{eq:MPOC_CC}.
	\end{proof}
	
	Proceeding in a similar way as used for the proof of \cref{prop:St_sur_SC}, we can validate
	the following result.
	\begin{proposition}\label{prop:St_sur_CC}
		Let $(\bar x,\bar y,\bar z)\in\R^n\times\R^q\times\R^q$ be feasible to \eqref{eq:MPOC_CC}
		and assume that the index sets $I^{0-}(\bar x)$ and $I^{-0}(\bar x)$ are empty.
		If $(\bar x,\bar y,\bar z)$ is C$_\textup{CC}$-stationary (M$_\textup{CC}$-stationary, 
		S$_\textup{CC}$-stationary) for \eqref{eq:MPOC_CC},
		then it is W-stationary	(M-stationary, S-stationary) for \eqref{eq:MPOC}.
	\end{proposition}
	
	In terms of \cref{prop:St_MPOC_to_St_CC,prop:St_sur_CC}, it seems to be promising to focus on
	the computation of stationary points associated with the complementarity-constrained program
	\eqref{eq:MPOC_CC} in order to find stationary points of \eqref{eq:MPOC}. Similarly to the
	switching-constrained approach described in \cref{sec:SC_ref}, we face the difficulty that the
	stationary points of the surrogate program \eqref{eq:MPOC_CC} do not always correspond to 
	stationary points of \eqref{eq:MPOC}. Thus, both approaches share the same qualitative properties.
	
	In order to solve \eqref{eq:MPOC_CC} computationally, it is possible to exploit e.g.\ problem-tailored
	SQP-methods, cf.\ \cite{FletcherLeyfferRalphScholtes2006,Leyffer2006}, or relaxation schemes, 
	see \cite{HoheiselKanzowSchwartz2013} for an overview. 
	Here, we focus on the well-known global relaxation approach of Scholtes, see \cite{Scholtes2001},
	which turned out to be numerical efficiency in comparison with other relaxation methods,
	see \cite{HoheiselKanzowSchwartz2013}.
	For some parameter $t\geq 0$, we consider the nonlinear surrogate problem
	\begin{equation}\label{eq:MPOC_CC_S}\tag{CC-MPOC$_\textup{S}(t)$}
		\begin{aligned}
			f(x)&\,\to\,\min\limits_{x,y,z}&&&\\
			g_i(x)&\,\leq\,0&\qquad&i\in\mathcal M&\\
			h_j(x)&\,=\,0&&j\in\mathcal P&\\
			G_l(x)-y_l&\,\leq\,0&&l\in\mathcal Q&\\
			H_l(x)-z_l&\,\leq\,0&&l\in\mathcal Q&\\
			y_l,z_l&\,\geq \,0&&l\in\mathcal Q&\\
			y_lz_l&\,\leq\,t&&l\in\mathcal Q&
		\end{aligned}
	\end{equation}
	which possesses $m+5q$ inequality and $p$ equality constraints.
	Noting that the feasible sets of \eqref{eq:MPOC_CC_S} form a nested family whose limit
	as $t\downarrow 0$ is the feasible set of \eqref{eq:MPOC_CC}, we can exploit the following
	strategy for the numerical solution of \eqref{eq:MPOC}.
	First, we choose a sequence $\{t_k\}_{k\in\N}$ of positive relaxation parameters converging
	to $0$. Afterwards, we use standard solvers from nonlinear programming to compute solutions
	associated with \hyperref[eq:MPOC_CC_S]{CC-MPOC$_\textup{S}(t_k)$}. The potential limit
	of this sequence is feasible to \eqref{eq:MPOC_CC} and, under some reasonable assumptions,
	a C$_\textup{CC}$-stationary point of this program, see \cite[Section~3.1]{HoheiselKanzowSchwartz2013}. 
	Due to \cref{prop:St_sur_CC}, this strategy is likely
	to produce W-stationary points of \eqref{eq:MPOC}. 
	
	At this point, we want to remark
	that the Scholtes-type relaxation approach from \cref{sec:SC_ref} seems to be numerically
	cheaper since the resulting relaxed surrogate program \eqref{eq:MPOC_SC_Scholtes} generally
	possesses less constraints than \eqref{eq:MPOC_CC_S}. On the other hand, due to the different
	role of the slack variables, the non-linearities in \eqref{eq:MPOC_CC_S} seem to be more
	balanced than in \eqref{eq:MPOC_SC_Scholtes}. A quantitative comparison of both methods is
	provided in \cref{sec:numerics}.

\section{Relaxation of or-constraints}\label{sec:relaxation}

In contrast to complementarity-, vanishing-, switching-, or cardinality-constrained 
programming where essential difficulties arise from the fact that the feasible set is almost
disconnected, or-constrained programs may behave geometrically well in this regard (apart
from pathological cases comprising e.g.\ optimization problems with gap domains, see \cref{sec:ex_gap_domain}).
However, we still need to deal with the combinatorial structure of the feasible set and the
irregularity at the or-kink. As we will see later, a direct treatment as described in
\Cref{sec:direct_treatment} struggles with this issue. Thus a nearby idea is 
a relaxation of this kink. Motivated by the computational results from 
\cite{HoheiselKanzowSchwartz2013,KanzowMehlitzSteck2019}, we perform a global relaxation
and smoothing of the kink using a Scholtes-type approach which is visualized in \cref{fig:Scholtes}.
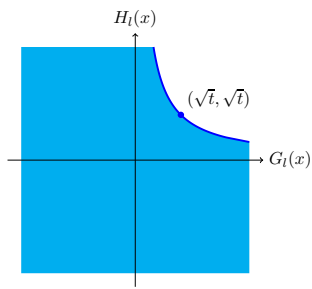
\begin{figure}[h]
\centering
\scalebox{.6}{\begin{tikzpicture}[domain=-2.5:2.5]
  \fill[color=cyan]  (-2.5,-2.5) -- (-2.5,2.5) foreach \x in {40,...,200}{  -- (\x/100,100/\x)} -- (2.5,0.4) -- (2.5,-2.5)
  		(-2.5,-2.5);
  \draw[->] (-2.8,0) -- (2.8,0) node[right] {$G_l(x)$};
  \draw[->] (0,-2.8) -- (0,2.8) node[above] {$H_l(x)$};
  \draw[very thick,color = blue]  (0.4,2.5) foreach \x in {40,...,200}{  -- (\x/100,100/\x)} -- (2.5,0.4);
   \fill[color=blue] (1,1) circle [radius=2pt];
  \node[above right] at (1,1){$(\sqrt t,\sqrt t)$};
 
\end{tikzpicture}}
\caption{Geometric illustration of the Scholtes-type global relaxation approach.}
\label{fig:Scholtes}
\end{figure}
Note that the popular relaxation approach due to Kanzow and Schwartz, 
see \cite{KanzowSchwartz2013,KanzowMehlitzSteck2019},
would only lead to a shift of the kink but preserves its difficult variational structure.
Thus, this idea does not reflect the general intention of this section which is why we do not
consider it here.

Let $t\geq 0$ be a relaxation parameter.
In order to perform the relaxation of our interest, we focus on two modified NCP-functions characterized below.
Note that any other (smoothed) or-compatible NCP-function can be used for this approach for the price of a 
potentially different underlying convergence analysis. 
\begin{itemize}
\item First, we will deal with the smoothed Fischer--Burmeister function 
$\varphi_\textup{FB}^t\colon\R^2\to\R$ given by
\[
	\forall (a,b)\in\R^2\colon\quad
	\varphi^t_\textup{FB}(a,b):=a+b-\sqrt{a^2+b^2+2t}.
\]
The smoothing of the Fischer--Burmeister function has been suggested by Kanzow in 
\cite{Kanzow1996} where $\varphi^t_\textup{FB}$ is used for the
numerical treatment of linear complementarity problems, see \cite{FukushimaLuoPang1998} as well.
The smoothing of NCP-functions in nonlinear complementarity-constrained programming is
the subject of interest in \cite{FacchineiJiangQi1999}.
\item For our second approach, we make use of $\varphi^t_\textup{KS}\colon\R^2\to\R$ given by
\[
	\forall (a,b)\in\R^2\colon\quad
	\varphi^t_\textup{KS}(a,b):=
	\varphi_\textup{KS}(a,b)-t
	=
	\begin{cases}
		ab-t		&\text{if }a+b\geq 0\\
		-\tfrac12(a^2+b^2+2t)	&\text{if }a+b<0.
	\end{cases}
\]
Clearly, this function is related to the NCP-function $\varphi_\textup{KS}$ from \cref{sec:direct_treatment}.
However, since $\varphi_\textup{KS}$ is already smooth, $\varphi^t_\textup{KS}$ cannot be referred to
as a smoothed NCP-function. Instead, $\varphi^t_\textup{KS}$ results from $\varphi_\textup{KS}$ by 
subtracting the offset $t$. In this way, the boundary of the associated zero sublevel set becomes smooth.
That is why we will refer to $\varphi^t_\textup{KS}$ as the
offset Kanzow--Schwartz function. Clearly, $\varphi^t_\textup{KS}$ is continuously
differentiable for each $t\geq 0$ since $\varphi_\textup{KS}$ possesses this property.
\end{itemize}

For some relaxation parameter $t\geq 0$ and a function 
$\varphi^t\in\{\varphi^t_\textup{FB},\varphi^t_\textup{KS}\}$, we now consider the relaxed surrogate
\begin{equation}\label{eq:MPOC_relaxed}\tag{P$(\varphi^t)$}
	\begin{aligned}
		f(x)&\,\rightarrow\,\min&&&\\
		g_i(x)&\,\leq\,0&\qquad&i\in\mathcal M&\\
		h_j(x)&\,=\,0&\qquad&j\in\mathcal P&\\
		\varphi^t(G_l(x),H_l(x))&\,\leq\,0&\qquad&l\in\mathcal Q&
	\end{aligned}
\end{equation}
whose feasible set will be denoted by $X(\varphi^t)$.
Noting that for each $t\geq 0$, one has
\[
	\forall (a,b)\in\R^2\colon\quad
	\varphi^t_\textup{FB}(a,b)\leq 0\,\Longleftrightarrow\,\varphi^t_\textup{KS}(a,b)\leq 0,
\]
the sets $X(\varphi^t_\textup{FB})$ and $X(\varphi^t_\textup{KS})$ are the same.
However, their particular \emph{nonlinear description} differs significantly.
In the lemma below, we summarize the geometrical properties of the family $\{X(\varphi^t)\}_{t\geq 0}$.
The proof of this result is rather standard and, thus, omitted.
\begin{lemma}\label{lem:geometry_relaxed_feasible_sets}
	For $\varphi^t\in\{\varphi^t_\textup{FB},\varphi^t_\textup{KS}\}$, the family
	$\{X(\varphi^t)\}_{t\geq 0}$ possesses the following properties:
	\begin{enumerate}
		\item $X(\varphi^0)=X$,
		\item $0\leq s\leq t\,\Longrightarrow\,X(\varphi^s)\subset X(\varphi^t)$, and
		\item $\bigcap_{t>0}X(\varphi^t)=X$.
	\end{enumerate}
\end{lemma}

Due to the above lemma, the following general strategy for the numerical treatment of
\eqref{eq:MPOC} is reasonable. For a sequence $\{t_k\}_{k\in\N}$ of positive relaxation
parameters converging to zero, we solve the relaxed surrogate 
\hyperref[eq:MPOC_relaxed]{P$(\varphi^{t_k}_\textup{FB})$} or
\hyperref[eq:MPOC_relaxed]{P$(\varphi^{t_k}_\textup{KS})$}.
Noting that these problems are standard nonlinear programs, it is reasonable to demand
that we are in position to compute associated KKT points.
If the obtained sequence of points possesses an accumulation point, then the latter is feasible to
\eqref{eq:MPOC}. In the following, we will discuss whether this accumulation point is stationary
for \eqref{eq:MPOC} as well. For that purpose, let us introduce the set
\[
	I^{\varphi^t}(x):=\{l\in\mathcal Q\,|\,\varphi^t(G_l(x),H_l(x))=0\}
\]
for a feasible point $x\in X(\varphi^t)$ of \eqref{eq:MPOC_relaxed}.
Clearly, $I^{\varphi^t}(x)$ comprises all indices corresponding to smoothed or-constraints
active at $x$.

\subsection{The smoothed Fischer--Burmeister function}

Here, we analyze the proposed relaxation scheme w.r.t.\ the 
smoothed Fischer--Burmeister function $\varphi^t_\textup{FB}$. 
First, we characterize its inherent convergence properties.
Afterwards, the regularity of the associated nonlinear
subproblems \hyperref[eq:MPOC_relaxed]{P$(\varphi^{t}_\textup{FB})$}
is discussed in more detail.
\begin{theorem}\label{thm:convergence_FB}
	Let $\{t_k\}_{k\in\N}$ be a sequence of positive relaxation parameters converging to zero.
	For each $k\in\N$, let $x_k\in X(\varphi^{t_k}_\textup{KS})$ be a KKT point of 
	\hyperref[eq:MPOC_relaxed]{P$(\varphi^{t_k}_\textup{FB})$}.
	Suppose that $\{x_k\}_{k\in\N}$ converges to some point $\bar x\in X$ where
	MPOC-MFCQ holds.
	Then, $\bar x$ is a W-stationary point of \eqref{eq:MPOC}.
\end{theorem}
\begin{proof}
	Since $x_k$ is a KKT point of \hyperref[eq:MPOC_relaxed]{P$(\varphi^{t_k}_\textup{FB})$} for each $k\in\N$, we find
	multipliers $\lambda^k_i\geq 0$ ($i\in I^g(x_k)$), $\rho^k_j$ ($j\in\mathcal P$), and $\xi^k_l\geq 0$ 
	($l\in I^{\varphi^{t_k}_\textup{FB}}(x_k)$) such that
	\begin{align*}
		0
	&
		=\nabla f(x_k)
		+\sum\limits_{i\in I^g(x_k)}\lambda_i^k\nabla g_i(x_k)
		+\sum\limits_{j\in\mathcal P}\rho_j^k\nabla h_j(x_k)\\
	&\qquad
		+\sum\limits_{l\in I^{\varphi^{t_k}_\textup{FB}}(x_k)}
			\xi^k_l\left(\alpha^k_l\nabla G_l(x_k)
				+\beta^k_l\nabla H_l(x_k)\right)
	\end{align*}
	holds where we used
	\[
		\alpha^k_l:=1-\frac{G_l(x_k)}{\sqrt{G_l^2(x_k)+H_l^2(x_k)+2t_k}}\qquad
		\beta^k_l:=1-\frac{H_l(x_k)}{\sqrt{G_l^2(x_k)+H_l^2(x_k)+2t_k}}
	\]
	for all $k\in\N$ and all $l\in\mathcal Q$.
	Noting that $x_k\to\bar x$ holds while all involved mappings are continuous, we may assume
	\begin{align*}
		\forall k\in\N\colon\quad
		I^g(x_k)\subset I^g(\bar x),\qquad 
		I^{\varphi^{t_k}_\textup{FB}}(x_k)\subset \mathcal I(\bar x).
	\end{align*}
	For each $k\in\N$, we formally set $\lambda^k_l:=0$ for all $l\in I^g(\bar x)\setminus I^g(x_k)$
	as well as $\xi^k_l:=0$ for all 
	$l\in\mathcal I(\bar x)\setminus I^{\varphi^{t_k}_\textup{FB}}(x_k)$. This yields
	\begin{equation}\label{eq:KKT_FB_smoothed}
	\begin{aligned}
		0
	&
		=\nabla f(x_k)
		+\sum\limits_{i\in I^g(\bar x)}\lambda_i^k\nabla g_i(x_k)
		+\sum\limits_{j\in\mathcal P}\rho_j^k\nabla h_j(x_k)\\
	&\qquad
		+\sum\limits_{l\in \mathcal I(\bar x)}
			\xi^k_l\left(\alpha^k_l\nabla G_l(x_k)
				+\beta^k_l\nabla H_l(x_k)\right).
	\end{aligned}
	\end{equation}
	
	Note that $\alpha^k_l,\beta^k_l\in(0,2)$ holds for all $k\in\N$ and $l\in\mathcal Q$.
	This means that the sequences $\{\alpha^k_l\}_{k\in\N}$ and $\{\beta^k_l\}_{l\in\N}$ converge
	w.l.o.g.\ to $\alpha_l\in[0,2]$ and $\beta_l\in[0,2]$ for each $l\in \mathcal Q$, respectively.
	By construction, we have $\alpha_l=1$ and $\beta_l=0$ for all $l\in I^{0+}(\bar x)$ while
	$\alpha_l=0$ and $\beta_l=1$ hold true for all $l\in I^{+0}(\bar x)$.
	For each $l\in \mathcal Q$, we have
	\[
		\alpha^k_l+\beta^k_l
		=
		1+1-\frac{G_l(x_k)+H_l(x_k)}{\sqrt{G_l^2(x_k)+H_l^2(x_k)+2t_k}}
		=
		1-\frac{\varphi^t_\textup{FB}(G_l(x_k),H_l(x_k))}{\sqrt{G_l^2(x_k)+H_l^2(x_k)+2t_k}}
		\geq
		1
	\]
	by feasibility of $x_k$ for \hyperref[eq:MPOC_relaxed]{P$(\varphi^{t_k}_\textup{FB})$}.
	Taking the limit, we particularly have $\alpha_l+\beta_l\geq 1$ for all $l\in I^{00}(\bar x)$,
	which yields that $\alpha_l$ or $\beta_l$ is positive for $l\in I^{00}(\bar x)$.
	
	Let us assume that the sequence $\{(\lambda^k_{I^g(\bar x)},\rho^k,\xi_{\mathcal I(\bar x)}^k)\}_{k\in\N}$
	is unbounded. We set
	\[
		\forall k\in\N\colon\quad
		(\tilde\lambda^k_{I^g(\bar x)},\tilde\rho^k,\tilde\xi^k_{\mathcal I(\bar x)})
		:=
		\frac{(\lambda^k_{I^g(\bar x)},\rho^k,\xi_{\mathcal I(\bar x)}^k)}
			{\Vert(\lambda^k_{I^g(\bar x)},\rho^k,\xi_{\mathcal I(\bar x)}^k)\Vert_2}.
	\]
	Thus, $\{(\tilde\lambda^k_{I^g(\bar x)},\tilde\rho^k,\tilde\xi^k_{\mathcal I(\bar x)})\}_{k\in\N}$
	is bounded and converges w.l.o.g.\ to some nonvanishing 
	$(\tilde\lambda_{I^g(\bar x)},\tilde\rho,\tilde \xi_{\mathcal I(\bar x)})$.
	Dividing \eqref{eq:KKT_FB_smoothed} by $\Vert(\lambda^k_{I^g(\bar x)},\rho^k,\xi_{\mathcal I(\bar x)}^k)\Vert_2$
	and taking the limit $k\to\infty$ while respecting the properties of the limits $\alpha,\beta\in\R^q$ as well
	as the continuous differentiability of all involved mappings, we come up with
	\begin{align*}
		0
		&=
		\sum\limits_{i\in I^g(\bar x)}\tilde\lambda_i\nabla g_i(\bar x)
			+\sum\limits_{j\in\mathcal P}\tilde\rho_j\nabla h_j(\bar x)\\
		&\qquad
			+\sum\limits_{l\in I^{0+}(\bar x)}\tilde\xi_l\nabla G_l(\bar x)
			+\sum\limits_{l\in I^{+0}(\bar x)}\tilde\xi_l\nabla H_l(\bar x)
			+\sum\limits_{l\in I^{00}(\bar x)}\tilde\xi_l(\alpha_l\nabla G_l(\bar x)+\beta_l\nabla H_l(\bar x)).
	\end{align*}
	Noting that $\tilde\lambda_i\geq 0$ ($i\in I^g(\bar x)$) and $\tilde\xi_l\geq 0$ ($l\in\mathcal I(\bar x)$)
	holds true, the validity of MPOC-MFCQ yields $\tilde\lambda_i=0$ ($i\in I^g(\bar x)$),
	$\tilde\rho_j=0$ ($j\in\mathcal P$), $\tilde \xi_l=0$ ($l\in I^{0+}(\bar x)\cup I^{+0}(\bar x)$), and
	$\tilde\xi_l\alpha_l=\tilde\xi_l\beta_l=0$ ($l\in I^{00}(\bar x)$). Since $\alpha_l$ or $\beta_l$ is 
	positive for each $l\in I^{00}(\bar x)$, we already have $\tilde \xi_l=0$ ($l\in \mathcal I(\bar x)$).
	Summarizing these observations, the multiplier 
	$(\tilde\lambda_{I^g(\bar x)},\tilde\rho,\tilde\xi_{\mathcal I(\bar x)})$
	vanishes which is a contradiction.
	
	Thus, $\{(\lambda^k_{I^g(\bar x)},\rho^k,\xi_{\mathcal I(\bar x)}^k)\}_{k\in\N}$ is bounded and converges
	w.l.o.g.\ to some multiplier $(\lambda_{I^g(\bar x)},\rho,\xi_{\mathcal I(\bar x)})$.
	Therefore, taking the limit in \eqref{eq:KKT_FB_smoothed} yields 
	\begin{align*}
		0
		&=
		\nabla f(\bar x)+
		\sum\limits_{i\in I^g(\bar x)}\lambda_i\nabla g_i(\bar x)
			+\sum\limits_{j\in\mathcal P}\rho_j\nabla h_j(\bar x)\\
		&\qquad
			+\sum\limits_{l\in I^{0+}(\bar x)}\xi_l\nabla G_l(\bar x)
			+\sum\limits_{l\in I^{+0}(\bar x)}\xi_l\nabla H_l(\bar x)
			+\sum\limits_{l\in I^{00}(\bar x)}\xi_l(\alpha_l\nabla G_l(\bar x)+\beta_l\nabla H_l(\bar x)).
	\end{align*}
	with $\lambda_i\geq 0$ ($i\in I^g(\bar x)$) and $\xi_l\geq 0$ ($l\in\mathcal I(\bar x)$).
	Finally, we set 
	\[\begin{aligned}
		&\forall l\in I^{0+}(\bar x)\cup I^{00}(\bar x)\colon&
		&\mu_l:=
			\begin{cases}
				\xi_l			&l\in I^{0+}(\bar x),\\
				\xi_l\alpha_l	&l\in I^{00}(\bar x)
			\end{cases}&\\
		&\forall l\in I^{+0}(\bar x)\cup I^{00}(\bar x)\colon&
		&\nu_l:=
			\begin{cases}
				\xi_l			&l\in I^{+0}(\bar x),\\
				\xi_l\beta_l	&l\in I^{00}(\bar x)
			\end{cases}&
	\end{aligned}\]
	in order to see that $\bar x$ is a W-stationary point of \eqref{eq:MPOC}.
\end{proof}

At the first glance, the result from \cref{thm:convergence_FB} seems to be 
comparatively weak when taking into account similar investigations for other
classes of disjunctive programs. On the other hand, the fact that the proposed
method produces W-stationary points actually means that local minimizers
of \eqref{eq:MPOC} which \emph{are} only W-stationary can be found by this approach.
Apart from that, the variational geometry of \eqref{eq:MPOC} suggests that
the biactive situation is rather artificial at local minimizers of \eqref{eq:MPOC},
and whenever the biactive set is empty, then all introduced stationarity notions
for \eqref{eq:MPOC} coincide.
 
The following simple example confirms that the results of \cref{thm:convergence_FB}
cannot be strengthened.
\begin{example}\label{ex:limits_of_smoothing_approach}
	Let us consider the simple or-constrained program
	\begin{equation}\label{eq:simple_example}
		\begin{split}
			\tfrac12(x_1-1)^2+\tfrac12(x_2-1)^2&\,\to\,\min\\
			x_1\,\leq\,0\,\lor\,x_2&\,\leq\,0.
		\end{split}
	\end{equation}
	Its globally optimal solutions are given by $(1,0)$ and $(0,1)$ and these
	points are S-stationary. Furthermore, the point $\bar x:=(0,0)$ is W-stationary
	but no local minimizer of \eqref{eq:simple_example}. 
	
	Let us consider the associated program \hyperref[eq:MPOC_relaxed]{P$(\varphi^{t}_\textup{FB})$}
	for some $t\in (0,1]$. One can easily check that $x(t):=(\sqrt t,\sqrt t)$ is a
	KKT point of the latter. Taking the limit $t\downarrow 0$, we have $x(t)\to\bar x$.
	Note that MPOC-LICQ is valid at $\bar x$.
	This means that we cannot strengthen the assertion of \cref{thm:convergence_FB}.
\end{example}

In order to guarantee that the local minimizers associated with the nonlinear program
\hyperref[eq:MPOC_relaxed]{P$(\varphi^{t}_\textup{FB})$} are KKT points, a constraint
qualification needs to be imposed on the latter problem. As we will show below, the
validity of MPOC-MFCQ at some feasible point of \eqref{eq:MPOC} implies that standard 
MFCQ is valid in a neighborhood of this point w.r.t.\ 
\hyperref[eq:MPOC_relaxed]{P$(\varphi^{t}_\textup{FB})$}.
This way, the assumptions of \cref{thm:convergence_FB} turn out to be quite natural.
Particularly, the need for KKT points associated with 
\hyperref[eq:MPOC_relaxed]{P$(\varphi^{t}_\textup{FB})$}
is not restrictive since MPOC-MFCQ is demanded to hold at the associated limit point.
\begin{proposition}\label{prop:regularity_FB}
	Let $\bar x\in X$ be a feasible point of \eqref{eq:MPOC} where MPOC-MFCQ is valid.
	Then, there  exists a neighborhood $U\subset\R^n$ of $\bar x$ such that MFCQ holds
	for \hyperref[eq:MPOC_relaxed]{P$(\varphi^t_\textup{FB})$} at all points from
	$X(\varphi^t_\textup{FB})\cap U$ for all $t>0$.
\end{proposition}
\begin{proof}
	Invoking \cite[Lemma~2.2]{KanzowMehlitzSteck2019}, we find a neighborhood $U$ of $\bar x$ such that the union
	\begin{align*}
		\Bigl[\{\nabla g_i(x)\,|\,i\in I^g(\bar x)\}
			&\cup\{\nabla G_l(x)\,|\,l\in I^{0+}(\bar x)\cup I^{00}(\bar x)\}\\
			&\cup\{\nabla H_l(x)\,|\,l\in I^{+0}(\bar x)\cup I^{00}(\bar x)\}\Bigr]
			\cup\{\nabla h_j(x)\,|\,j\in\mathcal P\}	
	\end{align*}
	is positive-linearly independent for each $x\in U$ since MPOC-MFCQ is valid and all appearing functions
	are continuously differentiable. 
	
	Now, fix $t>0$ as well as $x\in X(\varphi^t_\textup{FB})\cap U$.
	If $U$ is small enough, we have $I^g(x)\subset I^g(\bar x)$ and $I^{\varphi^t_\textup{FB}}(x)\subset \mathcal I(\bar x)$ 
	by continuity of $g$, $G$, $H$, and $\varphi^t_\textup{FB}$. Let us set
	\[
		\alpha_l:=1-\frac{G_l(x)}{\sqrt{G_l^2(x)+H_l^2(x)+2t}}\qquad
		\beta_l:=1-\frac{H_l(x)}{\sqrt{G_l^2(x)+H_l^2(x)+2t}}
	\]
	for all $l\in\mathcal Q$. By construction, it holds $\alpha_l,\beta_l\in(0,2)$. 
	Furthermore, $\alpha_l+\beta_l\geq 1$ holds for all $l\in\mathcal Q$
	since $x$ is feasible to  \hyperref[eq:MPOC_relaxed]{P$(\varphi^t_\textup{FB})$}. 
	Hence, $\alpha_l$ or $\beta_l$ is positive
	for each $l\in\mathcal Q$. Clearly, we have
	\[
		\begin{aligned}
			&\forall l\in I^{0+}(\bar x)\cap I^{\varphi^t_\textup{FB}}(x)\colon\quad& &\alpha_l\neq 0\quad& &\beta_l\approx 0&\\
			&\forall l\in I^{+0}(\bar x)\cap I^{\varphi^t_\textup{FB}}(x)\colon\quad& &\alpha_l\approx 0\quad& &\beta_l\neq 0
		\end{aligned}
	\]
	if $U$ is chosen sufficiently small.
	Thus, we may assume that the union
	\begin{equation}\label{eq:pos_lin_ind_FB}
		\begin{aligned}
			\Bigl[	&\{\nabla g_i(x)\,|\,i\in I^g(\bar x)\}\\
					&\cup\{\alpha_l\nabla G_l(x)+\beta_l\nabla H_l(x)\,|\,l\in (I^{0+}(\bar x)\cup I^{+0}(\bar x))\cap I^{\varphi^t_\textup{FB}}(x)\}\\
					&\cup\{\nabla G_l(x)\,|\,l\in I^{00}(\bar x)\cap I^{\varphi^t_\textup{FB}}(x)\}\\
					&\cup\{\nabla H_l(x)\,|\,l\in I^{00}(\bar x)\cap I^{\varphi^t_\textup{FB}}(x)\}\Bigr]
			\cup\{\nabla h_j(x)\,|\,j\in\mathcal P\}
		\end{aligned}
	\end{equation}
	is positive-linearly independent.
	
	Now, suppose that there are multipliers $\lambda_i\geq 0$ ($i\in I^g(x)$), $\rho_j$ ($j\in\mathcal P$), and $\xi_l\geq 0$
	($l\in I^{\varphi^t_\textup{FB}}(x)$) such that
	\begin{align*}
		0
		&=
		\sum\limits_{i\in I^g(x)}\lambda_i\nabla g_i( x)
			+\sum\limits_{j\in\mathcal P}\rho_j\nabla h_j(x)
			+\sum\limits_{l\in I^{\varphi^t_\textup{FB}}(x)}
			\xi_l\left(\alpha_l\nabla G_l(x)+\beta_l\nabla H_l(x)\right)
	\end{align*}
	is valid. This is equivalent to
	 \begin{align*}
		0
		&=
		\sum\limits_{i\in I^g(x)}\lambda_i\nabla g_i( x)
			+\sum\limits_{j\in\mathcal P}\rho_j\nabla h_j(x)\\
		&\qquad
			+\sum\limits_{l\in (I^{0+}(\bar x)\cup I^{+0}(\bar x))\cap I^{\varphi^t_\textup{FB}}(x)}
			\xi_l\left(\alpha_l\nabla G_l(x)+\beta_l\nabla H_l(x)\right)\\
		&\qquad
			+\sum\limits_{l\in I^{00}(\bar x)\cap I^{\varphi^t_\textup{FB}}(x)}
				\xi_l\alpha_l\nabla G_l(x)
			+\sum\limits_{l\in I^{00}(\bar x)\cap I^{\varphi^t_\textup{FB}}(x)}\xi_l\beta_l\nabla H_l(x)
	\end{align*}
	since we have $I^{\varphi^t_\textup{FB}}(x)\subset \mathcal I(\bar x)$
	by choice of $U$. The positive-linear independence of the union in \eqref{eq:pos_lin_ind_FB} 
	and $I^g(x)\subset I^g(\bar x)$ yield
	$\lambda_i=0$ ($i\in I^g(x)$), $\rho_j=0$ ($j\in \mathcal P$), $\xi_l=0$ 
	($l\in(I^{0+}(\bar x)\cup I^{+0}(\bar x))\cap I^{\varphi^t_\textup{FB}}(x)$), 
	$\xi_l\alpha_l=0$ ($l\in I^{00}(\bar x)\cap I^{\varphi^t_\textup{FB}}(x)$), as well as
	$\xi_l\beta_l=0$ ($I^{00}(\bar x)\cap I^{\varphi^t_\textup{FB}}(x)$).
	Noting that $\alpha_l$ or $\beta_l$ is positive, we can infer $\xi_l=0$ for all indices
	$l\in I^{00}(\bar x)\cap I^{\varphi^t_\textup{FB}}(x)$,
	i.e.\ $\xi_l=0$ holds for all $l\in I^{\varphi^t_\textup{FB}}(x)$.
	Consequently, MFCQ holds for \hyperref[eq:MPOC_relaxed]{P$(\varphi^t_\textup{FB})$}
	at $x$.	
\end{proof}

\subsection{The offset Kanzow--Schwartz function}

Now, we investigate the proposed relaxation scheme in terms of the function $\varphi^t_\textup{KS}$.
Recalling that the problems \hyperref[eq:MPOC_relaxed]{P$(\varphi^{t}_\textup{FB})$} and
\hyperref[eq:MPOC_relaxed]{P$(\varphi^{t}_\textup{KS})$} possess the same feasible sets, it is reasonable to
believe that the qualitative properties of this method to not significantly differ from the relaxation
approach involving the smoothed Fischer--Burmeister function $\varphi^t_\textup{FB}$.
In order to check this, let us review \cref{ex:limits_of_smoothing_approach} first. 
Indeed, it is not difficult to see that $x(t):=(\sqrt t,\sqrt t)$ is a KKT point of the program
\hyperref[eq:MPOC_relaxed]{P$(\varphi^{t}_\textup{KS})$} associated with \eqref{eq:simple_example}
for each $t\in(0,1]$ again. Since we have $x(t)\to\bar x$ as $t\downarrow 0$ where $\bar x:=(0,0)$ is
a W-stationary point of \eqref{eq:simple_example} where MPOC-LICQ holds, the above conjecture seems to be
confirmed.

Fix $t>0$ and some $\tilde x\in X(\varphi^t_\textup{KS})$. Assume that $I^{\varphi^t_\textup{KS}}(\tilde x)$ is nonempty.
Then, for each $l\in I^{\varphi^t_\textup{KS}}(\tilde x)$, the mapping $x\mapsto \varphi^t_\textup{KS}(G_l(x),H_l(x))$
behaves bilinear w.r.t.\ $G_l(x)$ and $H_l(x)$ for arguments from a neighborhood of $\tilde x$ since $G_l$ and $H_l$
are continuous functions.
Particularly, the relaxed subproblem \hyperref[eq:MPOC_relaxed]{P$(\varphi^{t}_\textup{KS})$} corresponds to
a classical Scholtes-type relaxation locally around $\tilde x$ in the sense of the particular underlying 
nonlinear description of the relaxed feasible set.
That is why the proofs of the upcoming results, which characterize the convergence behavior of the suggested
relaxation scheme as well as the regularity of the associated nonlinear subproblems, directly follow by 
reprising the arguments used in \cite[Section~3.1]{HoheiselKanzowSchwartz2013} and
\cite[Section~3]{KanzowMehlitzSteck2019} in the context of MPCCs and MPSCs, respectively, while doing some
problem-tailored but nearby adjustments. 
\begin{theorem}\label{thm:convergence_KS}
	Let $\{t_k\}_{k\in\N}$ be a sequence of positive relaxation parameters converging to zero.
	For each $k\in\N$, let $x_k\in X(\varphi^{t_k}_\textup{KS})$ be a KKT point of 
	\hyperref[eq:MPOC_relaxed]{P$(\varphi^{t_k}_\textup{KS})$}.
	Suppose that $\{x_k\}_{k\in\N}$ converges to some point $\bar x\in X$ where
	MPOC-MFCQ holds.
	Then, $\bar x$ is a W-stationary point of \eqref{eq:MPOC}.
\end{theorem}
\begin{proposition}\label{prop:regularity_KS}
	Let $\bar x\in X$ be a feasible point of \eqref{eq:MPOC} where MPOC-MFCQ is valid.
	Then, there  exists a neighborhood $U\subset\R^n$ of $\bar x$ such that MFCQ holds
	for \hyperref[eq:MPOC_relaxed]{P$(\varphi^t_\textup{KS})$} at all points from
	$X(\varphi^t_\textup{KS})\cap U$ for all $t>0$.
\end{proposition}

Due to the above results, the qualitative properties of the proposed relaxation scheme do not depend on the
actual choice of the underlying function from $\{\varphi^t_\textup{FB},\varphi^t_\textup{KS}\}$.
However, we note that the nonlinearities hidden within these two functions are essentially different which
is why we want to investigate the quantitative properties of the respective resulting relaxation method in numerical
practice, see \cref{sec:numerics}.

\section{Numerical results}\label{sec:numerics}

In this section, we are going to compare the solution approaches discussed in 
\cref{sec:direct_treatment,sec:disjunctive_modification,sec:relaxation} by means
of different instances of or-constrained programming. 
Particularly, we are going to investigate the direct replacement of the or-constraints
by means of nonlinear inequalities induced by the Kanzow--Schwartz function $\varphi_\textup{KS}$,
see \cref{sec:direct_treatment}, the reformulation of the or-constrained program
as an MPSC or MPCC which then is treated with the aid of suitable Scholtes-type 
relaxation methods, see \cref{sec:disjunctive_modification}, and the direct
Scholtes-type relaxation approach based on the smoothed Fischer--Burmeister function
$\varphi_\textup{FB}^t$ and the offset Kanzow--Schwartz function $\varphi^t_\textup{KS}$
discussed in \cref{sec:relaxation}. 
The following problems, which are chosen from model classes with significant practical relevance, 
will serve as the benchmark for our numerical comparison:
\begin{enumerate}
	\item a nonlinear disjunctive program in the sense of Balas, see \cref{sec:ex_disjunctive_programming},
	\item an optimization problem where the domains of the underlying variables possess gaps, 
		see \cref{sec:ex_gap_domain}, and
	\item an or-constrained optimal control problem of the non-stationary heat equation in two
		spacial dimensions, see \cref{sec:ex_optimal_control}.
\end{enumerate}

For each of these examples, we first discuss the underlying problem structure. Afterwards, the 
numerical results are presented. In order to present a reasonable quantitative comparison of the
five discussed computational methods, we make use of performance profiles, see \cite{DolanMore2002},
based on computed function values. Note that we do not use time as an performance index here since
the transformation of the or-constrained program into an MPSC or MPCC comes for the cost of several
slack variables and additional constraints whose respective number depends linearly on the number of
original or-constraints. Thus, we can expect that the other approaches would clearly outrun these 
two methods w.r.t.\ computation time. 
In order to guarantee that the nonlinear surrogate programs which arise from the different solution 
methods we want to compare can be tackled with the same NLP solver, we decided only to use the smooth
Kanzow--Schwartz NCP-function for the direct reformulation of the or-constraints, cf.\ \cref{sec:direct_treatment}. 
Furthermore, we would like to mention that the use of other
relaxation methods for MPSCs and MPCCs, see \cite{KanzowMehlitzSteck2019,HoheiselKanzowSchwartz2013},
is possible when using the approach from \cref{sec:disjunctive_modification} but, as it turned out, 
does not yield results that differ significantly from those ones obtained via the Scholtes-type
relaxations. 

\subsection{Implementation}

The subsequently described numerical experiments were carried out using MATLAB R2018a.
For our comparison, we exploited the five algorithms stated below:
\begin{itemize}[leftmargin=8em]
	\item[\textbf{IPOPT}:] the IPOPT interior-point algorithm from \cite{WaechterBiegler2006}
		is applied to the NLP which results from \eqref{eq:MPOC} by reformulating all
		or-constraints with the aid of the smooth Kanzow--Schwartz function, see \cref{sec:direct_treatment},
	\item[\textbf{ScholtesSC}:] the Scholtes-type relaxation method which is applied to
		a switching-constrained reformulation \eqref{eq:MPOC_SC} of \eqref{eq:MPOC}, see \cref{sec:SC_ref},
	\item[\textbf{ScholtesCC}:] the relaxation method of Scholtes is applied to a
		complementarity-constrained reformulation \eqref{eq:MPOC_CC} of \eqref{eq:MPOC}, see \cref{sec:ref_CC},
	\item[\textbf{smoothedFB}:] the direct relaxation method from \cref{sec:relaxation}
		using the smoothed Fischer--Burmeister function, and
	\item[\textbf{offsetKS}:] the direct relaxation method from \cref{sec:relaxation}
		which exploits the offset Kanzow-Schwartz function.
\end{itemize}
Each of these algorithms is called via user-supplied gradients of objective and constraint functions.
We use the global stopping tolerance $10^{-4}$ for IPOPT's stopping tolerance in case of algorithm
\textbf{IPOPT} and for the maximum or-constraint violation
\[
	\max\{\max\{0,\min\{G_l(x),H_l(x)\}\}\,|\,l\in\mathcal Q\}
\]
in case of the other four methods.
In order to allow a comparison of the computational results, the relaxed subproblems arising in the
methods \textbf{ScholtesSC}, \textbf{ScholtesCC}, \textbf{smoothedFB}, and \textbf{offsetKS} are
solved with IPOPT as well. Here, the internal stopping tolerance of IPOPT is set to $10^{-6}$. 
For all these relaxation
approaches, the relaxation parameter is chosen to be $t_k:=0.01^k$ for each $k\in\mathbb N$, and the
algorithm is automatically terminated whenever $t_k$ drops below $10^{-8}$. 

Since we aim for a fair quantitative comparison of these five methods, we cannot rely on computation
time since by construction, the numerical effort of these approaches is essentially different.
Instead, we focus our attention on the comparison of computed function values (w.r.t.\ different starting 
points) with the globally optimal function value in order to classify 
the robustness of the suggested methods. In light of the fact that or-constrained
programs are likely to possess a substantial amount of local minimizers which are not globally optimal,
this is a reasonable approach. Here, we make use of the quantity
\begin{equation}\label{eq:performance_metric}
	Q_\delta(x^a_s):=
		\begin{cases}
			f(x^a_s)-f_\textup{min}+\delta	&\text{if }x^a_s\text{ is feasible within tolerance},\\
			+\infty &\text{otherwise}
		\end{cases}
\end{equation}
as the underlying metric for the resulting performance profiles. 
Above, we used $x^a_s$ in order to denote the final iterate of a run of algorithm $a\in\mathcal A$ with
\[
	\mathcal A:=\{\textbf{IPOPT},\textbf{ScholtesSC},\textbf{ScholtesCC},\textbf{smoothedFB},\textbf{offsetKS}\}
\]
for the starting point associated with the index $s\in\mathcal S$. If unknown, a reasonable approximate
of the global minimal function value $f_\textup{min}$ needs to be determined. Finally, $\delta\geq 0$
is an additional parameter which reduces sensitivity to numerical accuracy.
Using the metric $Q_\delta$ defined above, the resulting performance ratio is given by
\[
	\forall s\in\mathcal S\,\forall a\in\mathcal A\colon\quad
		r_{s,a}:=\frac{Q_\delta(x^a_s)}{\min\{Q_\delta(x^\alpha_s)\,|\,\alpha\in\mathcal A\}}.
\]
In our performance profiles, we plot the illustrative parts of the curves $\rho_a\colon[1,\infty)\to[0,1]$
given by
\[
	\forall \tau\in[1,\infty)\colon\quad
		\rho_a(\tau):=\frac{\abs{\{s\in\mathcal S\,|\,r_{s,a}\leq\tau\}}}{\abs{\mathcal S}}
\]
for each algorithm $a\in\mathcal A$ where $\abs{\cdot}$ denotes the cardinality of a set.
Thus, $\rho_a(\tau)$ may be interpreted as the probability that the final iterate produced by algorithm $a$ 
has a function value which is not worse than $\tau$-times the best computed function value w.r.t.\ all
algorithms from $\mathcal A$.

\subsection{Numerical experiments}

In this section, we present the numerical results associated with three prominent instances of
or-constrained programming. 

\subsubsection{Disjunctive programming}\label{sec:ex_disjunctive_programming}

Let us define sets $X_1,X_2\subset\R^3$ as stated below:
\begin{align*}
	X_1&:=\{x\in\R^3\,|\,x_1\geq 4,\,x_1+(x_2-2)^2+(x_3+2)^2\geq 5\},\\
	X_2&:=\{x\in\R^3\,|\,x_1^2+x_2^2\leq x_3,\,(x_1-1)^2+x_2^2+x_3\geq 1,\,x_2\leq 0\}.
\end{align*}
Now, we consider the nonlinear program
\begin{align*}
	(x_1-1)^2+(x_2-2)^2+(x_3+2)^2&\,\to\,\min\\
	x&\,\in\,X_1\cup X_2
\end{align*}
which can be interpreted as an instance of disjunctive programming in the sense of Balas, see \cite{Balas2018}.
One can easily check that its global minimizer is given by $\bar x:=(0,0,0)$ which possesses the minimal
function value $f_\textup{min}=9$. Note that this program possesses additional local minimizers which are not globally
optimal at all points from the set $\{(1,0,1)\}\cup\{(4,x_2,x_3)\in\R^3\,|\,(x_2-2)^2+(x_3+2)^2=1\}$. 
Introducing two slack variables $u,v\in\R$, we can equivalently restate the program of interest as the
or-constrained problem
\begin{equation}\label{eq:ex_disjunctive}
	\begin{split}
		(x_1-1)^2+(x_2-2)^2+(x_3+2)^2&\,\to\,\min\limits_{x,u,v}\\
			4-x_1-u&\,\leq\,0\\
			5-x_1-(x_2-2)^2-(x_3+2)^2-u&\,\leq\,0\\
			x_1^2+x_2^2-x_3-v&\,\leq\,0\\
			1-(x_1-1)^2-x_2^2-x_3-v&\,\leq\,0\\
			x_2-v&\,\leq\,0\\
			u\,\leq\,0\,\lor\,v&\,\leq\,0
	\end{split}
\end{equation}
which can be processed by our five algorithms. We use $500$ starting points whose $x$-components are
randomly chosen from $[0,4]$ while $u$ and $v$ are random scalars from $[-1,0]$.
The resulting performance profile for $\delta:=1$ can be found in \cref{fig:PerfProf_disjunctive}.
As we can see, the relaxation methods reliably compute the best function value and identify
the actual global minimizer in most of the cases. There is no significant difference between
the direct relaxation methods and those ones which are applied to surrogate reformulations of
\eqref{eq:ex_disjunctive}. All these algorithms do not outrun \textbf{IPOPT}
whose performance is also quite good since it finds the best function value in more than $85\%$
of the cases. This is, however, not surprising since problem \eqref{eq:ex_disjunctive} possesses just
one or-constraint.

\begin{figure}[h]\centering
\includegraphics[width=8.0cm]{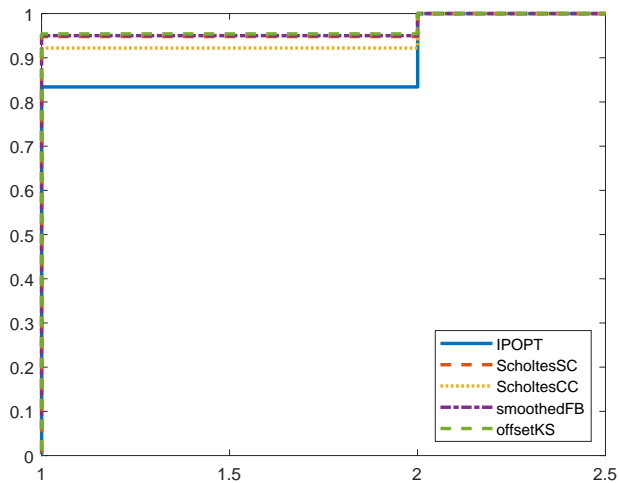}
\caption{Performance profile for the disjunctive program from \cref{sec:ex_disjunctive_programming}.}
\label{fig:PerfProf_disjunctive}
\end{figure}

\subsubsection{Optimization problems with gap domains}\label{sec:ex_gap_domain}

In contrast to standard box-constrained programming, it may happen that variables need to be chosen
such that they \emph{do not} belong to a critical interval. One may think of physical quantities needing to
stay away from given critical values or situations in production planning where a certain amount of products 
has to be bought \emph{or} sold. 
In order to model such constraints, we fix vectors $\ell,u\in\R^n$ satisfying $\ell<u$ and 
consider the system
\begin{equation}\label{eq:gap_domain}
	x_l\leq \ell_l\,\lor\,x_l\geq u_l\qquad l=1,\ldots,n.
\end{equation}
These constraints induce so-called \emph{gap domains} which are heavily disconnected. 
Here, the feasible set crumbles into $2^n$ branches.
Consequently, the underlying optimization problem is likely to possess several local minimizers which are
not globally optimal. We note that due to $\ell<u$, the biactive set $I^{00}(x)$ is empty for all
feasible points $x$ of the underlying or-constrained optimization problems. 
This means that all the introduced stationarity notions, see \cref{def:stationarities:MPOC}, coincide
for programs with or-constraints of type \eqref{eq:gap_domain}.

For a random vector $a\in[0,1]^{50}$ sorted in ascending order with at least $15$ entries which are greater than $0.5$, 
we consider the optimization problem 
\begin{equation}\label{eq:ex:gap_domain}
	\begin{aligned}
		\mathsmaller\sum\nolimits_{l=1}^{50}(x_l-a_l)^2&\,\to\,\min&&&\\
		\mathsmaller\sum\nolimits_{l=1}^{50}x_l&\,\leq\,15&&&\\
		x_l\,\leq\,0\,\lor\,x_l&\,\geq\,1&\qquad&l=1,\ldots,50
	\end{aligned}
\end{equation}
whose variables possess gap domains. By construction, its globally minimal function value is given by
\[
	f_\textup{min}=\mathsmaller\sum\nolimits_{l=1}^{35}a_l^2+\mathsmaller\sum\nolimits_{l=36}^{50}(1-a_l)^2.
\]

For our experiments, we challenged our algorithms with $500$ randomly chosen starting points from $[-1,2]^{50}$.
Two resulting performance profiles for $\delta:=1$ with differently scaled $\tau$-axes 
can be found in \cref{fig:PerfProf_gap_domain}.
\begin{figure}[h]\centering
\includegraphics[width=6.5cm]{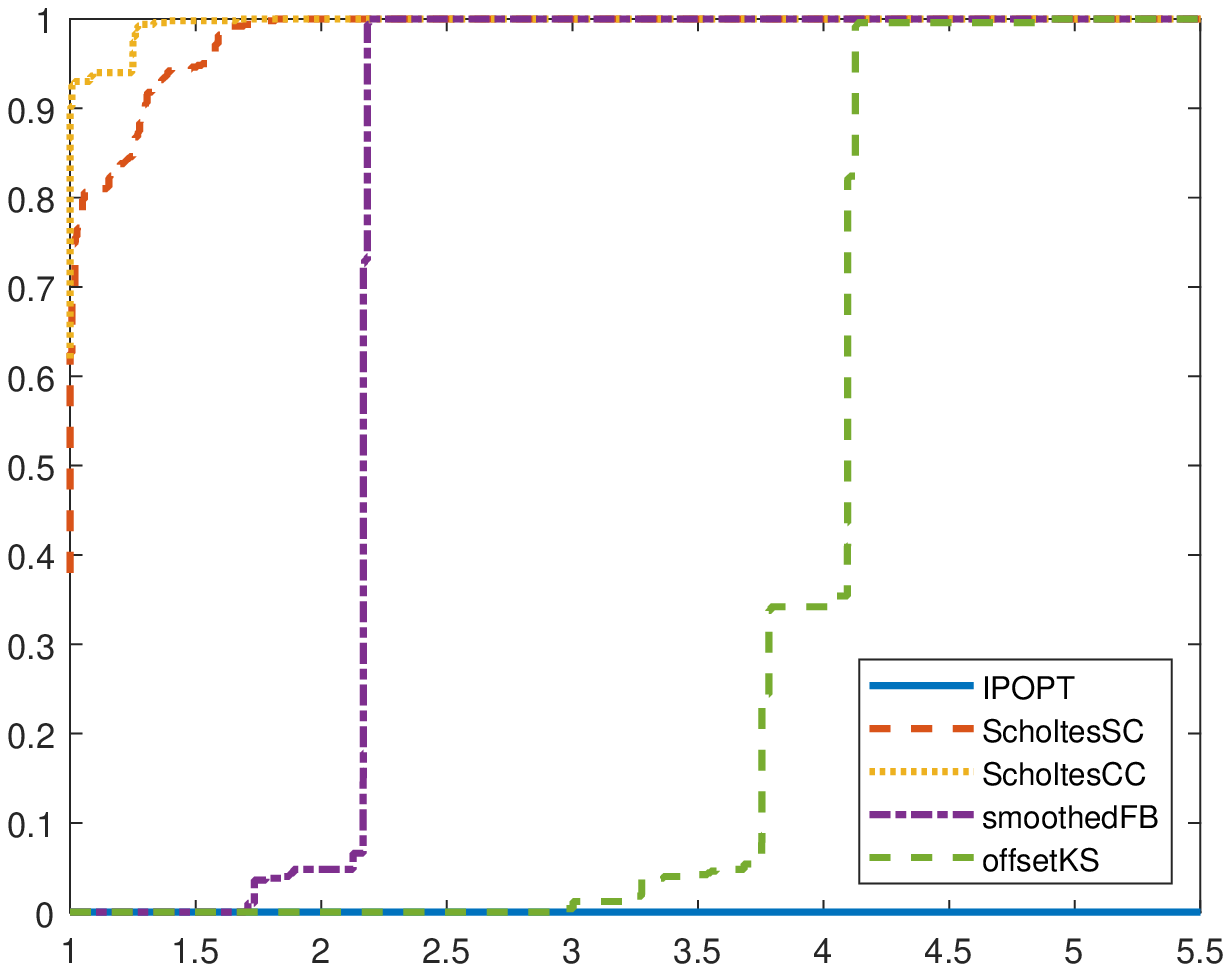}
\hspace{1em}
\includegraphics[width=6.5cm]{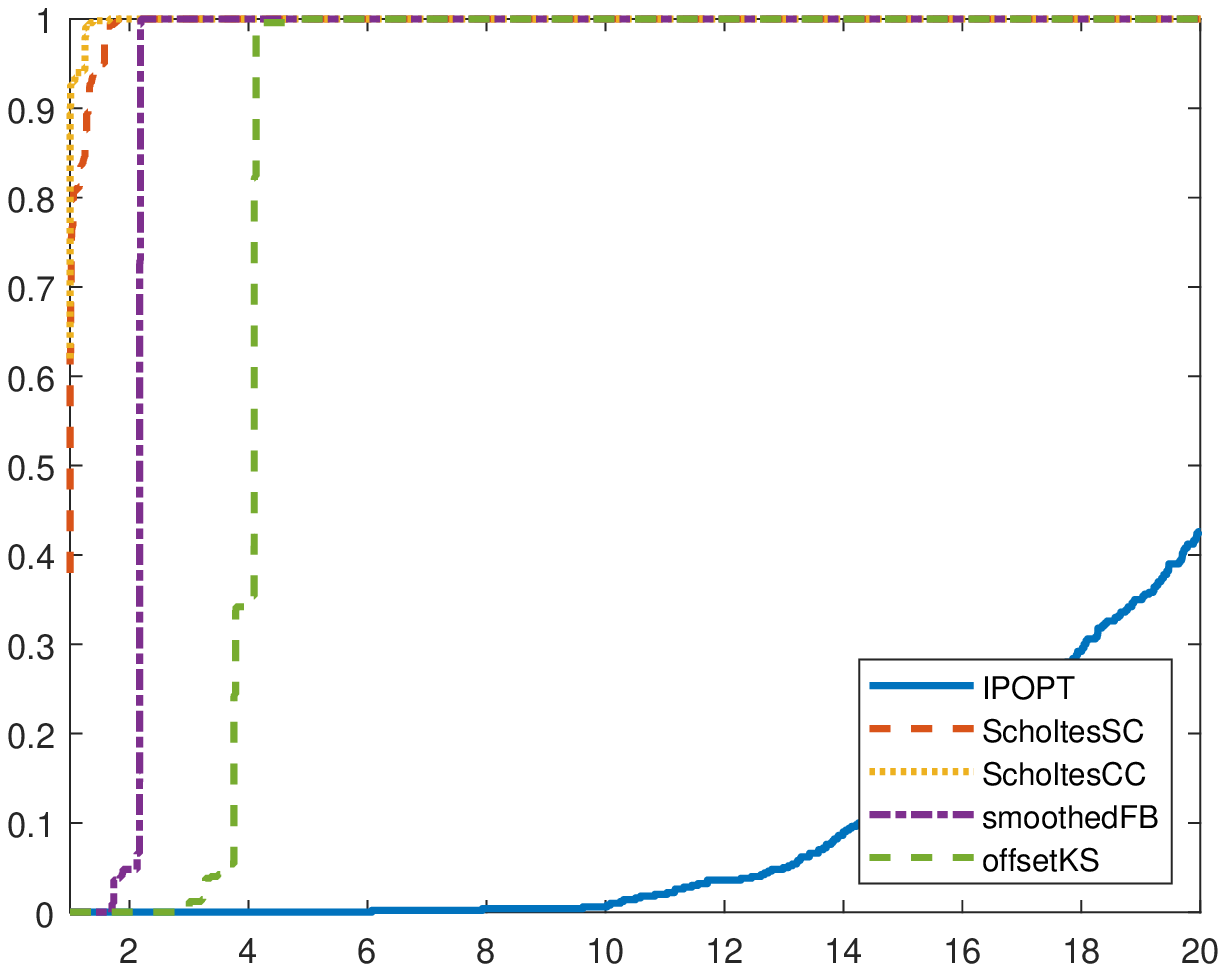}
\caption{Performance profiles for the optimization problem with gap domains from \cref{sec:ex_gap_domain}.}
\label{fig:PerfProf_gap_domain}
\end{figure}

Noting that the feasible set of \eqref{eq:ex:gap_domain} is disconnected, it is not surprising that the
relaxation methods clearly outrun \textbf{IPOPT} which generally gets stuck in the \emph{branch} of
\eqref{eq:ex:gap_domain} associated with the respective starting point. 
A direct relaxation of the program by means of
\textbf{smoothedFB} or \textbf{offsetKS} does not really solve this issue. For example, one can easily check
that the method \textbf{offsetKS} relaxes the gap-constraints to
\[
	x_1(1-x_1)\leq t\qquad l=1,\ldots,n
\]
which is equivalent to the or-constrained system
\[
	x_1\leq\tfrac12-\sqrt{\tfrac14-t}\,\lor\,x_1\geq\tfrac12+\sqrt{\tfrac14-t}
\]
for each $t\in[0,\tfrac14]$. That means that the method needs to handle highly disconnected feasible sets
for comparatively large relaxation parameters already. Similar effects can be observed for \textbf{smoothedFB}. Both direct
relaxation methods turn out to compute one particular locally optimal solution which is not the global
minimizer of \eqref{eq:ex:gap_domain} in most of the situations, respectively, and the performance
profiles underline this observation. 
It is not difficult to see that the MPSC- or MPCC-reformulations of \eqref{eq:ex:gap_domain} considered in 
\cref{sec:disjunctive_modification} still possess disconnected feasible sets.
However, the associated Scholtes-type relaxation methods \textbf{ScholtesSC} and \textbf{ScholtesCC}
seem to be much more stable in numerical practice since they
compute the actual global minimizer of \eqref{eq:ex:gap_domain} in most of the situations.
One reason for this behavior might be the presence of slack variables which allow some freedom when the nonlinear subproblems
are solved.

\subsubsection{Or-constrained optimal control}\label{sec:ex_optimal_control}

Motivated by the considerations in \cite[Section~6.2.2]{KanzowMehlitzSteck2019}, we want to 
study the optimal control of the non-stationary heat equation with the aid of two control
functions $u$ and $v$ which influence distinct parts $\Omega_u$ and $\Omega_v$ 
of the underlying domain $\Omega$ over time. Here, we 
additionally assume that at least one of the controls needs to be nonnegative at each
time instance. Such a constraint arises when due to technical restrictions, it is not
possible to cool $\Omega_u$ and $\Omega_v$ at the same time.
In terms of this paper, this means that the control functions need to satisfy an or-constraint in 
pointwise fashion. In order to guarantee that the associated optimal control problem possesses
an optimal solution, standard $L^2$-regularity of controls is generally not enough since this 
conservative regularity assumption does not guarantee the weak sequential closedness of the
underlying set of feasible controls. Following ideas from \cite{ClasonRundKunisch2017,ClasonDengMehlitzPruefert2019}
where pointwise switching  or complementarity constraints are considered, this issue can be
solved by considering controls from a first-order Sobolev space.

Fix $I:=(0,6)$, $\Omega:=(-1,1)^2$, and let $\Gamma$ be the boundary of $\Omega$. 
Furthermore, we set $\Omega_u:=(-1,0]\times(-1,1)$ and $\Omega_v:=(0,1)\times(-1,1)$. 
The non-stationary heat equation of our interest is given by
\begin{equation}\label{eq:state_equation}
	\begin{aligned}
		\partial_ty(t,\omega)-\Delta_\omega y(t,\omega)
			-\tfrac1{10}\chi_{\Omega_u}(\omega)u(t)-\tfrac1{10}\chi_{\Omega_v}(\omega)v(t)
			&\,=\,0&\qquad&\text{a.e.\ on }I\times\Omega&\\
		\vec{\mathbf n}(\omega)\cdot\nabla _\omega y(t,\omega)&\,=\,0
			&&\text{a.e.\ on }I\times\Gamma&\\
			y(0,\omega)&\,=\,0&&\text{a.e.\ on }\Omega
	\end{aligned}
\end{equation}
where $\chi_A\colon\Omega\to\R$ denotes the characteristic function of the measurable set $A\subseteq\Omega$
which equals $1$ on $A$ and vanishes on $\Omega\setminus A$. 
Following classical arguments, see \cite{Troeltzsch2009} where the Lebesgue and Sobolev spaces of interest are characterized as well, 
there exists a continuous linear mapping
$S\colon H^1(I)\times H^1(I)\to L^2(I;H^1(\Omega))$ which assigns to each pair $(u,v)$ of controls the
uniquely determined (weak) solution $y$ of \eqref{eq:state_equation}. 
Let us define the desired state $y_\textup{d}:=S(u_\textup{d},v_\textup{d})$ where 
$u_\textup{d},v_\textup{d}\in H^1(I)$ are given by
\begin{equation}\label{eq:desired_controls}
	\forall t\in I\colon\quad
	u_\textup{d}(t):=-20\,\sin(\pi t/3)\qquad
	v_\textup{d}(t):= 10\,\cos(\pi t/2).
\end{equation}
Now, we are in position to state the optimal control problem of our interest below:
\begin{equation}\label{eq:ex_optimal_control}
	\begin{aligned}
		\frac12\norm{S(u,v)-y_{\textup d}}{L^2(I;L^2(\Omega))}^2
		+\frac{\alpha}{2}\left(\norm{u}{L^2(I)}^2+\norm{v}{L^2(I)}^2\right)&&&\\
		+\frac{\beta}{2}\left(\norm{\partial_tu}{L^2(I)}^2+\norm{\partial_tv}{L^2(I)}^2\right)
			&\,\to\,\min\limits_{u,v}&&&\\
			u(t)\,\geq\,0\,\lor\,v(t)&\,\geq\,0&\qquad&\text{a.e.\ on }I.
	\end{aligned}
\end{equation}
For our experiments, we choose $\alpha:=10^{-6}$ and $\beta:=10^{-5}$.
Observe that the pair $(u_\textup{d},v_\textup{d})$ is not feasible to \eqref{eq:ex_optimal_control}
since these functions violate the pointwise or-constraint precisely for all those $t\in I$ satisfying $1< t< 3$.

In order to tackle \eqref{eq:ex_optimal_control} with the suggested algorithms, we first need to perform a
suitable discretization. Therefore, we tessellate the domain $\Omega$ with the aid of the function
\texttt{generateMesh} from MATLAB's PDE toolbox using the tolerance $h:=10^{-1}$. 
The time interval $I$ is subdivided into equidistant intervals of width $\vartheta:=5\cdot 10^{-2}$. 
Noting that state and control need to possess first-order Sobolev regularity, we use standard piecewise
affine and continuous finite elements for spatial and temporal discretization. This leads to a conforming
approximation of the $H^1$-norm in the objective functional of \eqref{eq:ex_optimal_control}.

This discretization results in a finite-dimensional program of type \eqref{eq:MPOC} which possesses $121$
simple or-constraints on the discretized control functions and a convex, quadratic objective functional.
Thus, this program can be decomposed into $2^{121}$ convex subproblems which indicates that the overall
discretized program possesses a huge amount of local minimizers. For our comparison of the suggested
numerical methods, we need to identify a reasonable candidate for a \emph{global} minimizer of the
optimal control problem. In order to do this, we use the following heuristic procedure adapted from 
\cite[Section~6.2.2]{KanzowMehlitzSteck2019} in order to
find a coarse upper bound for the globally minimal function value. First, we solve the program 
\emph{exactly} for a rough time discretization (we used $\vartheta:=0.375$) by computing the (global) minimizers
of all $2^{17}$ resulting convex subproblems and comparing the obtained solutions. Afterwards, we lift the
obtained global minimizer to the finer time grid using linear interpolation. The obtained point is used as a starting
point for our five algorithms. The best obtained outcome possesses a function value of $f_\textup{min}=0.2156$.
The resulting controls are depicted in \cref{fig:PerfProf_optimal_control}. They are closely related to
$u_\textup{d}$ and $v_\textup{d}$ from \eqref{eq:desired_controls} except for the time interval $(1,3)$ where the pointwise 
or-constraint from \eqref{eq:ex_optimal_control} leads to significant changes.
\begin{figure}[h]\centering
\includegraphics[width=6.5cm]{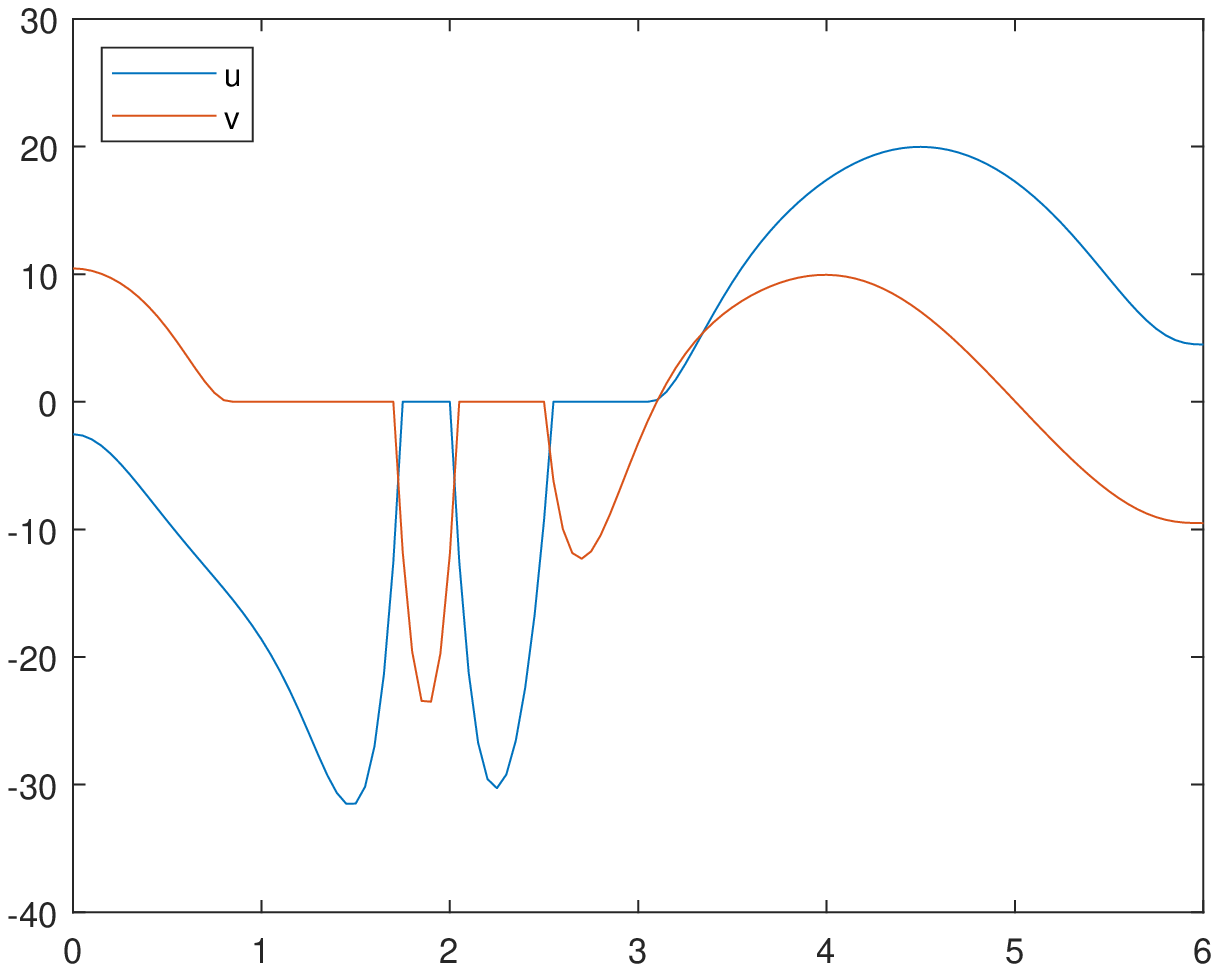}
\hspace{1em}
\includegraphics[width=6.5cm]{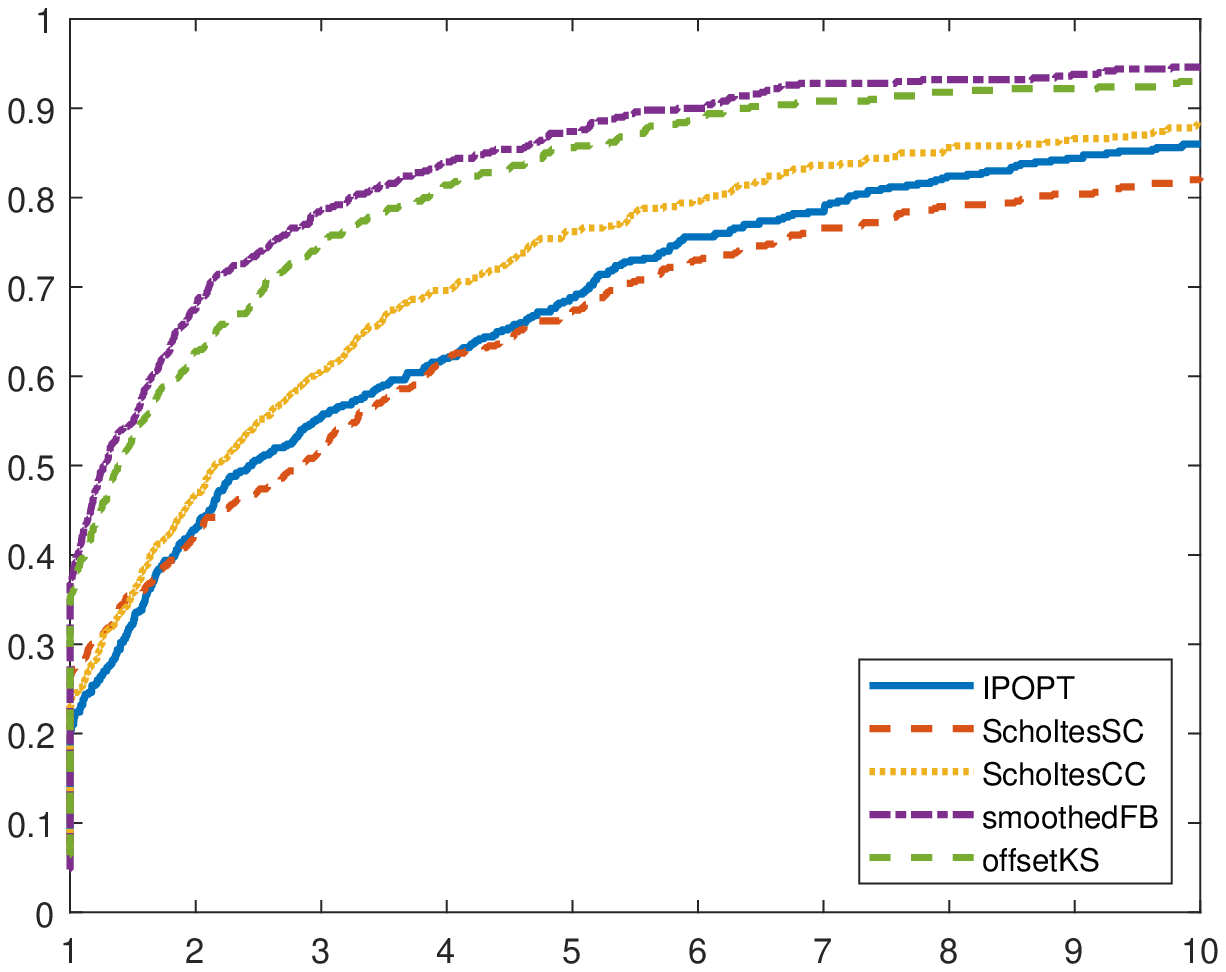}
\caption{Potential global minimizer (left) and performance profile (right) 
	for the or-constrained optimal control problem from \cref{sec:ex_optimal_control}.}
\label{fig:PerfProf_optimal_control}
\end{figure}

For our numerical experiment, we performed algorithmic runs for $500$ starting points which were
randomly chosen elementwise from $[-10,10]$. The resulting performance profile for $\delta:=0$
can be found in \cref{fig:PerfProf_optimal_control}. As it turns out, the Scholtes-type direct
relaxation methods \textbf{smoothedFB} and \textbf{offsetKS} perform much better than the other two
relaxation methods \textbf{ScholtesSC} and \textbf{ScholtesCC}. A reason for that might be that due to
the transformation to a switching- or complementarity-constrained program, the surrogate problems under
consideration in \textbf{ScholtesSC} and \textbf{ScholtesCC} possess lots of additional slack variables, 
namely $242$, and inequality constraints which makes them uncomfortably large. 
Due to the fact that the feasible set of the discretized or-constrained optimal control problem is strongly
connected, the direct method \textbf{IPOPT} keeps up at least with the latter relaxation methods. 
Another reason for that behavior might be the fact that MPOC-LICQ is valid at all feasible points of
the discretized optimal control problem which implies that GCQ holds at all feasible points of
the associated surrogate \hyperref[eq:MPOC_NCP]{\textup{MPOC}$(\varphi_\textup{KS})$}, 
see \cref{lem:MPOC_LICQ_yields_GCQ_KS}.
However, \textbf{IPOPT} cannot challenge the direct relaxation methods \textbf{smoothedFB} and \textbf{offsetKS} which
produce points with the best objective value much more frequently. Finally, it should be noted that
\textbf{smoothedFB} performs slightly better than \textbf{offsetKS}. This might be caused by the fact
that the relaxation via the smoothed Fischer--Burmeister function avoids bilinearities which appear
when the Kanzow--Schwartz function is used for that purpose.

\subsection{Summary}

Our examples indicate that the \emph{correct} choice for a numerical method which can be used to solve
or-constrained optimization problems heavily depends on the underlying problem structure.
In situations where only a few or-constraints need to be considered while the resulting feasible set is
still connected, there is no significant difference between all the suggested algorithms, 
see \cref{sec:ex_disjunctive_programming}.
On the other hand, optimization problems with gap domains should be transferred into surrogate
MPSCs or MPCCs which then should be solved by classical relaxation methods. This procedure turned out
to annihilate the disconnectedness of the underlying feasible set successfully, see \cref{sec:ex_gap_domain}. 
Finally, whenever a
huge number of simple or-constraints needs to be considered such that the underlying feasible set is
still connected, then a direct relaxation of the program seems to be the correct approach since this
approach regularizes the feasible set while not blowing up the number of variables and constraints, 
see \cref{sec:ex_optimal_control}.

\section{Concluding remarks}\label{sec:concluding_remarks}

In this paper, we discussed three different approaches for the numerical handling of or-constrained 
optimization problems with the aid of first-order methods from continuous optimization.
First, we investigated the reformulation of or-constraints as (smooth or nonsmooth) inequality
constraints using suitable NCP-functions. Second, we transferred the or-constrained optimization
problem into a switching- or complementarity-constrained surrogate problem which can be 
solved numerically with the aid of relaxation methods. The qualitative
properties of these transformations were discussed in detail. Third, a direct Scholtes-type relaxation of
optimization problems with or-constraints based on the smoothed Fischer--Burmeister function or the
offset Kanzow--Schwartz function was suggested and the convergence properties of this approach were
investigated. A numerical comparison of all these methods based on different models from or-constrained
optimization has been carried out. It turned out that the precise choice of the method heavily depends
on the structural properties of the underlying problem's feasible set. Generally, relaxation methods
perform much better than algorithms based on a simple replacement of the or-constraints using NCP-functions.

\end{document}